\documentclass[11pt,reqno]{amsart}

\usepackage[T1]{fontenc}
\usepackage{lmodern}
\usepackage{microtype}
\usepackage{amsmath,amssymb,mathtools}
\usepackage{mathrsfs}
\usepackage{enumitem}
\usepackage{xcolor}
\usepackage{tikz}
\usetikzlibrary{arrows.meta,calc}
\usepackage[hidelinks,pdfusetitle]{hyperref}

\allowdisplaybreaks
\numberwithin{equation}{section}

\newtheorem{theorem}{Theorem}[section]
\newtheorem{proposition}[theorem]{Proposition}
\newtheorem{lemma}[theorem]{Lemma}

\theoremstyle{definition}

\theoremstyle{remark}
\newtheorem{remark}[theorem]{Remark}

\newcommand{\R}{\mathbb R}
\newcommand{\T}{\mathbb T}
\newcommand{\C}{\mathbb C}
\newcommand{\D}{\mathbb D}
\newcommand{\one}{\mathbf 1}
\newcommand{\eps}{\varepsilon}
\newcommand{\Om}{\Omega}
\newcommand{\pa}{\partial}
\newcommand{\cF}{\mathcal F}
\newcommand{\cG}{\mathcal G}

\newcommand{\Rea}{\operatorname{Re}}

\title[Vortex patches with arbitrarily many genuine holes]
{Uniformly Rotating Vortex Patches with Arbitrarily Many Genuine Holes}
\author[Z. Xue]{Zhilong Xue}
\address{Academy of Mathematics and Systems Science, Chinese Academy of
Sciences, Beijing 100190, People's Republic of China}
\email{zhilongxue@amss.ac.cn}

\author[W. Zhan]{Weicheng Zhan}
\address{School of Mathematical Sciences, Xiamen University, Xiamen 361005,
People's Republic of China}
\email{zhanweicheng@amss.ac.cn}

\begin{document}

\begin{abstract}
For every prescribed integer $N\ge2$, we construct uniformly rotating
unit-vorticity vortex patches $\omega=\one_D$ for the planar Euler equation
such that $D$ is connected and $\R^2\backslash D$ has exactly $N$ bounded
connected components.  These components are genuine zero-vorticity holes,
rather than opposite-sign vortex patches or regions carrying a second
nonzero vorticity level.  The angular velocities lie in the
rigidity-compatible interval $(0,1/2)$, and the domains converge in measure
to the Rankine disk as the holes collapse.

The construction starts from a fixed co-rotating polygonal configuration of
$N$ unit-vorticity vortex patches and removes a shrinking spatial copy of
that configuration from the Rankine disk.  An exact complement identity
solves all inner-boundary equations before the outer circle is perturbed.
The remaining defect is generated by the $N$-th exterior multipole and has
size $\varepsilon^{N+2}$.  The resulting outer correction feeds back into
the normalized inner problem at size $\varepsilon^{2N}$.  Separate renormalization of the outer and inner equations produces a limiting affine system with
a lower-triangular derivative.  Its diagonal blocks are the nonresonant
Rankine operator and the angular-velocity-augmented linearization of the
fixed seed configuration.  We also determine the first corrections to the
outer boundary, the hole boundaries, and the angular velocity.


\bigskip

\end{abstract}

\maketitle

\section{Introduction and main results}

\subsection{Vortex patches and \texorpdfstring{$V$}{V}-states}

We consider the two-dimensional incompressible Euler equation in vorticity form
\begin{equation}\label{eq:euler}
  \pa_t\omega+u\cdot\nabla\omega=0,
  \qquad
  u=\nabla^\perp\psi,
  \qquad
  \psi=\frac1{2\pi}\log|\cdot|*\omega,
\end{equation}
where $\nabla^\perp=(-\pa_2,\pa_1)$.  If $E\subset\R^2$ is bounded, we write
\begin{equation}\label{eq:001}
  \psi_E(x)=\frac1{2\pi}\int_E\log|x-y|\,\mathrm{d}y,
  \qquad
  u_E=\nabla^\perp\psi_E.
\end{equation}
We denote by $R_\vartheta$ the counterclockwise rotation through angle
$\vartheta$.  Global existence and uniqueness of classical solutions of the
two-dimensional Euler equation go back to Wolibner and H\"older
\cite{Wol33,Hol33}.  At the weak level, Yudovich proved global
well-posedness for bounded, integrable vorticity
\cite{Yud63}.  A bounded unit-vorticity vortex patch is a
vorticity of the form $\omega_0=\one_{D_0}$ with $D_0\subset\R^2$ bounded.
Such a patch is
transported by the unique Yudovich flow.  If $\pa D_0$ is of class $C^{k,\gamma}$, with $k\ge1$ and
$0<\gamma<1$, then this boundary regularity persists globally by the
results of Chemin and Bertozzi--Constantin \cite{Che93,BC93}; see also
\cite{MB02} for a systematic account of the vorticity formulation.
At the $C^2$ level, however, Kiselev and Luo established
ill-posedness by constructing $C^2$ initial patches whose boundary curvature
becomes unbounded at arbitrarily small positive times \cite{KL23}.

We say that the vortex patch $\omega=\one_D$ is uniformly rotating, or is a
$V$-state, with angular velocity $\Om$ if
\[
 \omega(t,x)=\one_D(R_{-\Om t}x)
\]
is a solution of \eqref{eq:euler}.  The corresponding stationary equation in
the rotating frame is
\begin{equation}\label{eq:002}
  (u_D(x)-\Om x^\perp)\cdot n(x)=0,
  \qquad x\in\pa D, 
\end{equation}
where $n(x)$ is the outward unit normal to the patch boundary $\pa D$ at $x$.
Equivalently, the relative stream function
\begin{equation}\label{eq:003}
  \Psi_{D,\Om}(x)=\psi_D(x)-\frac{\Om}{2}|x|^2
\end{equation}
is constant on each connected component of $\pa D$.  The constants on distinct
components need not agree.

The Rankine disk is the basic radial example.  Kirchhoff's ellipses provide
the classical noncircular family of simply connected $V$-states
\cite{Kir76}.  The numerical computations of Deem and Zabusky
\cite{DZ78} led to Burbea's local bifurcation theorem from the
disk \cite{Bur82}.  Subsequent work established boundary regularity and
analyticity of simply connected $V$-states
\cite{HMV13,CCG16}, bifurcation from
Kirchhoff ellipses \cite{HM16a}, and global continuation of the
Burbea branches \cite{HMW20}.  Earlier partial
rigidity results for rotating patches were obtained by Hmidi
\cite{Hmi15}.  Rigidity results impose
strong restrictions at the ends of the angular-velocity range: for compactly
supported unit-vorticity vortex patches, a nonradial $V$-state must have
$0<\Om<1/2$ \cite{GPSY21}.  Fan, Wang, and Zhan
extended this rigidity to vortex patches whose boundaries are merely
finite unions of Jordan curves and, more generally, to compactly supported,
possibly sign-changing multipatch vorticities \cite{FWZ25}.  See
also the weak-topology rigidity theorem near the Rankine vortex in the simply
connected class
\cite{Hua25}.
Computer-assisted methods have also produced analytic, simply connected,
sixfold-symmetric nonconvex $V$-states beyond the previously known explicit
families and local bifurcation constructions
\cite{CG25}.

For doubly connected $V$-states the radial states are annuli.  Nonannular
branches were constructed in \cite{dHMV16}, and
degenerate bifurcation regimes were analyzed in
\cite{HM16b,HR17,WXZ24}.  Related rigidity of
elliptic interfaces was proved in \cite{HMV15}.  There is
also a substantial theory in bounded radial fluid domains: steady vortex patches in
a disk were studied analytically and numerically in
\cite{dHHM16}, while a recent general framework
covers, among other models, Euler $V$-states in disks, annuli, and exteriors
of disks \cite{HXX26b}.  Variational constructions of
rotating vortex patches in a disk were obtained in
\cite{CWWZ21}, and sharp rigidity for disk patches whose
boundaries are finite unions of Jordan curves was established in
\cite{FWZ24}.

Another line of work concerns several disjoint vortical components.  Early
computations include the equal co-rotating vortex pairs studied by Saffman and
Szeto
\cite{SS80} and Dritschel's equilibria of two to eight uniform
vortices \cite{Dri85}.  An early general desingularization
result for nondegenerate point-vortex systems is due to Wan
\cite{Wan88}.  Rigorous desingularization constructions now
produce co-rotating and counter-rotating vortex-patch configurations from
nondegenerate point-vortex equilibria
\cite{Tur85,HM17,Gar21,HW22}; local
vortex-pair branches have also been continued globally
\cite{GH23}.  Configurations with several disjoint doubly
connected vortical components were obtained in \cite{CQZZ21}.
Among these results, Garc\'ia's regular $N$-polygon family
\cite{Gar21} supplies the inner template used in our proof.  Its vortical
region, however, is disconnected.  For a recent survey of periodic vortex
patches and further references, see \cite{Gar26}.

\subsection{Genuine holes and the relation to earlier computations}

The distinction between a genuine hole and an additional vorticity level is
essential.  We seek domains of the form
\begin{equation}\label{eq:004}
 D=G\setminus\bigcup_{j=1}^N\overline{H_j},
 \qquad
 \omega=\one_D,
\end{equation}
where $G$ is simply connected, the $H_j$ are pairwise disjoint and simply connected, and
$\overline{H_j}\Subset G$.  Equivalently,
\begin{equation}\label{eq:005}
 \one_D=\one_G-\sum_{j=1}^N\one_{H_j}.
\end{equation}
The coefficient of every subtracted component is exactly one.  This is the
feature that turns a co-rotating polygonal configuration of positive vortex
patches into a collection of zero-vorticity holes.  All boundary curves are
therefore interfaces of a single characteristic-function vorticity and are
coupled through the same nonlocal velocity field.  By contrast, a union of
disjoint positive vortex patches has a disconnected vortical region, whereas
in a multilevel configuration
\[
 \one_G-\gamma\sum_{j=1}^N\one_{H_j}
\]
the vorticity in an inner component equals $1-\gamma$ and need not vanish.
The problem considered here has neither of these two structures.
The present use of the word ``hole'' should also be distinguished from
rotating hollow vortices, or $H$-states.  Those are potential-flow
free-boundary configurations with circulation around a hollow core, rather
than characteristic-function vorticities of the form $\one_D$; see
\cite{CNK21}.

Numerical evidence for configurations with genuine holes is more limited.
Wu, Overman, and Zabusky developed contour-dynamics algorithms for rotating
and translating piecewise-constant $V$-states \cite{WOZ84}, and Overman
subsequently investigated the local geometry of limiting states, including
doubly and triply connected configurations \cite{Ove86}. These works do not
establish a rigorous existence theorem for smooth triply connected
$V$-states. A particularly relevant dynamical computation is due to Velasco
Fuentes \cite{VF13}, who evolved two symmetric holes inside a Rankine vortex,
including the genuine zero-vorticity case. In an elastic-interaction regime,
the holes rotate around one another with nearly uniform angular velocity,
while their shapes and separation undergo small oscillations; the resulting
motion is therefore not an exact relative equilibrium.

More generally, genuinely non-rigid recurrent vortex patch dynamics have
recently been constructed rigorously. These include time quasi-periodic
patches near Kirchhoff ellipses \cite{BHM23}, quasi-periodic motions near
annular patches with one hole \cite{HHR24}, and time-periodic leapfrogging
motions of concentrated planar vortex patches \cite{HHM25}. These solutions
have geometries different from those considered here. Our result is
complementary: it constructs exact uniformly rotating relative equilibria
with an arbitrary prescribed number of genuine zero-vorticity holes.

A closely related recent rigorous construction is the work of Baroncini, Cantero,
Garc\'ia, Hassainia, and Mateu \cite{BCGHM26}.  They construct a
rotating outer vortex patch containing several highly concentrated inner
components.  Their vorticity has the form
\[
 \one_{D_1^\eps}-\frac1{\pi\eps^2}
 \sum_{j=1}^N\one_{D_{2,j}^\eps}.
\]
Thus the vorticity on an inner component equals
$1-(\pi\eps^2)^{-1}$ in their local regime and is not zero.  Their analytical
and numerical configurations are consequently multilevel, sign-changing
states rather than unit-vorticity characteristic-function patches of
the form $\omega=\one_D$.
They also identify continuation to the parameter at which this inner value
vanishes as an open global problem.  Our construction follows a different
local route: the coefficient in \eqref{eq:005} is
exactly one from the beginning.

A connected radially symmetric vortex patch domain is necessarily a disk or
an annulus and therefore has at most one bounded complementary component.
Hence no connected vortex patch domain with $N\ge2$ separated holes is
radial, so there is no radial state from which to bifurcate.
Theorem~\ref{thm:main} constructs such states directly by coupling a
co-rotating polygon at the inner scale to a perturbed Rankine boundary at the
outer scale.  To the best of our knowledge, for every prescribed
$N\ge2$ this is the first rigorous construction for the planar Euler equation of a bounded, connected, uniformly rotating characteristic-function vortex patch
$\omega=\one_D$ such that $\R^2\setminus D$ has exactly $N$ bounded
connected components.
The resulting angular velocities lie in the rigidity-compatible interval
$(0,1/2)$.

\subsection{Main results}

We identify $\R^2$ with $\C$.  For $N\ge2$, put
\[
 \zeta=e^{2\pi\mathrm{i}/N},
\]
and define the order-$2N$ dihedral group
\[
 D_N:=\{z\mapsto\zeta^jz,\ z\mapsto\zeta^j\bar z:
 0\le j\le N-1\}.
\]
Thus a set $E\subset\C$ is $D_N$-invariant when
$\zeta E=E$ and $\bar E=E$; the analogous convention is used for scalar
functions. With this notation, our main existence result is the following.

\begin{theorem}\label{thm:main}
Let $N\ge2$, let $\ell>0$, and let $\alpha\in(0,1)$.  There exist a
microscopic parameter $\rho_*>0$, a fixed $D_N$-invariant union
\[
 P=P_1\cup\cdots\cup P_N
\]
of pairwise disjoint simply connected $C^{2-\alpha}$ domains, with
fundamental boundary chart $p_0$, numbers $\delta\in(0,1/(2N))$,
$\Om_*=1/2-\delta$, and $\eps_0>0$, and families
\[
 \bigl\{G_\eps,H_{1,\eps},\ldots,H_{N,\eps},\Om_\eps\bigr\}_{0<\eps<\eps_0}.
\]
The vorticity $\one_P$ is a co-rotating configuration of $N$ unit-vorticity
vortex patches with angular velocity $\delta$.  The families above
have the following properties.
\begin{enumerate}[label=\textup{(\roman*)}]
\item The sets $G_\eps,H_{1,\eps},\ldots,H_{N,\eps}$ are simply connected
$C^{2-\alpha}$ domains.  The closures of the $H_{j,\eps}$ are pairwise
disjoint and are compactly contained in $G_\eps$.
\item The domain
\begin{equation}\label{eq:006}
  D_\eps
  =G_\eps\setminus
   \bigcup_{j=1}^N\overline{H_{j,\eps}}
\end{equation}
is connected, and $\R^2\setminus D_\eps$ has exactly $N$ bounded connected
components.  Equivalently, $\widehat{\C}\setminus D_\eps$ has exactly
$N+1$ connected components.
\item The domain $D_\eps$ is $D_N$-invariant.
\item One has $\Om_\eps\in(0,\frac12)$ and
\begin{equation}\label{eq:007}
  \omega_\eps(t,x)
  =\one_{D_\eps}(R_{-\Om_\eps t}x)
\end{equation}
is a global Yudovich solution of \eqref{eq:euler}.
\item There are a radial profile $R_\eps$ and an inner boundary chart
$p_\eps$ such that
\[
 \pa G_\eps
 =\{R_\eps(\theta)e^{\mathrm{i}\theta}:\theta\in\T\},
 \qquad
 \eps^{-1}\pa H_{j,\eps}
 =\{\zeta^{j-1}p_\eps(w):w\in\T\},
\]
as in Section~\ref{sec:01}.  The maps
$\eps\mapsto R_\eps$, $\eps\mapsto p_\eps$, and $\eps\mapsto\Om_\eps$
extend continuously to $\eps=0$, with limiting values $1$, $p_0$, and
$\Om_*$, respectively.  In particular, $R_\eps\to1$ and
$p_\eps\to p_0$ in $C^{2-\alpha}$, while $\Om_\eps\to\Om_*$, as
$\eps\downarrow0$.
\end{enumerate}
In particular, $N=2$ gives a triply connected uniformly rotating vortex patch.
\end{theorem}
In addition to existence, the construction determines the first
nontrivial corrections to the geometry and the angular velocity.
\begin{proposition}
\label{prop:01}
For the family in Theorem~\ref{thm:main}, there are a real number
$\lambda_*$, a function $h_*\in X_P$ (with $X_P$ defined in
\eqref{eq:XP}), and a constant $a_N>0$ such that, as $\eps\downarrow0$,
\begin{align}
 R_\eps(\theta)
 &=1+a_N\eps^{N+2}\cos(N\theta)
   +O_{C^{2-\alpha}}(\eps^{N+4}),
 \label{eq:008}\\
 \pa H_{j,\eps}
 &=\left\{\eps\zeta^{j-1}
 \left[p_0(w)+\rho_*^2\eps^{2N}h_*(w)
 +o_{C^{2-\alpha}}(\eps^{2N})\right]:w\in\T\right\},
 \label{eq:009}\\
 \Om_\eps
 &=\Om_*+\lambda_*\eps^{2N}+o(\eps^{2N}).
 \label{eq:010}
\end{align}
Here
\begin{equation}\label{eq:main-aN}
 a_N=\frac{M_N(P)}{\pi(1-2N\delta)},
 \qquad
 M_N(P):=\int_P z^N\,\mathrm{d}A(z)>0.
\end{equation}
Thus, to leading order, the maxima of the outer-boundary correction lie on
the rays through the polygonal conformal centers.  More explicitly, for some
$h_{\rho_*}\in X_P$,
\[
 p_0(w)=\ell+\rho_*\bigl(w+\rho_*h_{\rho_*}(w)\bigr).
\]
The conformal center and conformal radius mean, respectively, the constant
term and the modulus of the leading coefficient in the normalized exterior
Laurent chart.  In this sense, $H_{j,\eps}$ has conformal center
$\eps\ell\zeta^{j-1}$ and conformal radius $\eps\rho_*$.  Finally, as
$\eps\downarrow0$,
\begin{equation}\label{eq:011}
 |D_\eps\mathbin\triangle\D|
 =\eps^2|P|+4a_N\eps^{N+2}+O(\eps^{N+4}).
\end{equation}
\end{proposition}

\begin{remark}
\label{rem:01}
The two small parameters appearing above play different roles in the
construction. 
Proposition~\ref{prop:02} supplies arbitrarily small nonresonant
microscopic values $\rho>0$.  We first choose one such value $\rho_*$.
This fixes $P$, $p_0$, $\delta$, the inverse bounds of the two diagonal
linearized operators, and all constants in the gluing argument.  For this fixed
seed there is an $\eps_0=\eps_0(\rho_*)>0$ such that the conclusions above
hold for every $0<\eps<\eps_0$.  Thus only $\eps$ tends to zero in the
solution family, and no estimate uniform as $\rho_*\to0$ is used.

If a sequence of nonresonant seeds is chosen with $\rho_*\downarrow0$, then
\begin{equation}\label{eq:012}
 \delta=\frac{N-1}{4\ell^2}\rho_*^2+O(\rho_*^3),
 \qquad
 M_N(P)=N\pi\ell^N\rho_*^2+O(\rho_*^4),
\end{equation}
and consequently
\begin{equation}\label{eq:013}
 \Om_*=\frac12-\frac{N-1}{4\ell^2}\rho_*^2+O(\rho_*^3),
 \qquad
 a_N=N\ell^N\rho_*^2+O(\rho_*^4).
\end{equation}
\end{remark}

The geometry described in Theorem~\ref{thm:main} is illustrated schematically
in Figure~\ref{fig:01}.

\begin{figure}[!t]
\centering
\begin{tikzpicture}[
  x=1cm,y=1cm,
  line cap=round,
  line join=round,
  >={Latex[length=2mm,width=1.35mm]},
  boundary/.style={line width=.75pt,draw=black},
  leader/.style={-{Latex[length=1.8mm,width=1.2mm]},thin},
  every node/.style={font=\small,inner sep=1.5pt}
]
  \def\Rout{2.62}
  \def\outeramp{0.042}
  \def\centerrad{0.93}
  \def\holerad{0.235}

  \path[fill=black!6]
    plot[domain=0:360,samples=241,variable=\t]
      ({\Rout*(1+\outeramp*cos(5*\t))*cos(\t)},
       {\Rout*(1+\outeramp*cos(5*\t))*sin(\t)}) -- cycle;
  \draw[boundary]
    plot[domain=0:360,samples=241,variable=\t]
      ({\Rout*(1+\outeramp*cos(5*\t))*cos(\t)},
       {\Rout*(1+\outeramp*cos(5*\t))*sin(\t)}) -- cycle;

  \foreach \j in {0,...,4}{
    \pgfmathsetmacro{\ang}{72*\j}
    \path[fill=white]
      plot[domain=0:360,samples=81,variable=\t]
       ({\centerrad*cos(\ang)
          +\holerad*(1+0.055*cos(2*(\t-\ang)))*cos(\t)},
        {\centerrad*sin(\ang)
          +\holerad*(1+0.055*cos(2*(\t-\ang)))*sin(\t)}) -- cycle;
    \draw[boundary]
      plot[domain=0:360,samples=81,variable=\t]
       ({\centerrad*cos(\ang)
          +\holerad*(1+0.055*cos(2*(\t-\ang)))*cos(\t)},
        {\centerrad*sin(\ang)
          +\holerad*(1+0.055*cos(2*(\t-\ang)))*sin(\t)}) -- cycle;
  }

  \draw[-{Latex[length=2.25mm,width=1.5mm]},line width=.7pt]
    (24:3.04) arc[start angle=24,end angle=142,radius=3.04];
  \node at (84:3.29) {$\Om_\eps>0$};

  \node[align=right] (bulk) at (-1.05,1.53)
    {$D_\eps\;(\omega=1)$};
  \draw[leader] (bulk.south east) -- (-.57,1.08);

  \node[anchor=west,align=left] (hole) at (1.56,.97)
    {$H_{j,\eps}\;(\omega=0)$};
  \draw[leader] (hole.west) -- (.88,.29);

  \node[align=center] (outer) at (0,-3.02)
    {$\partial G_\eps:\quad
      R_\eps(\theta)=1+a_N\eps^{N+2}\cos(N\theta)
      +O(\eps^{N+4})$};
  \draw[leader] (outer.north east)
    to[out=30,in=-76] (1.98,-1.55);
\end{tikzpicture}
\caption{Schematic geometry of the solutions, shown with $N=5$.  The shaded
region has vorticity one and the white components are genuine zero-vorticity
holes.  The holes and the boundary deformation are enlarged for visibility
and are not drawn to scale.}
\label{fig:01}
\end{figure}
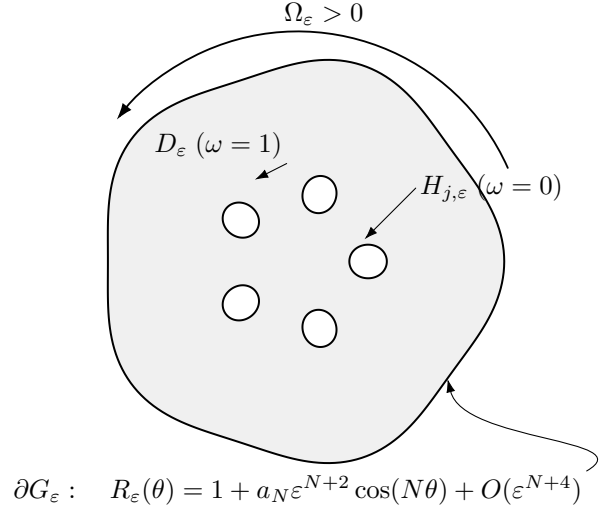

\subsection{\texorpdfstring{Proof strategy and the two scales}
{Proof strategy and the two scales}}

\begingroup
The construction begins with Garc\'ia's nondegenerate $D_N$-symmetric
co-rotating polygonal vortex-patch configuration \cite{Gar21}.  We use
amplitude homogeneity to normalize its vorticity to one, select and fix a
nonresonant microscopic value $\rho_*$, and then embed the resulting template
with the independent spatial parameter $\eps$.  Spatial dilation does not
change its angular velocity $\delta$.

The geometric core is the exact complement identity
\begin{equation}\label{eq:137}
 u_{\D\setminus\eps P}(\eps y)
 -\Bigl(\frac12-\delta\Bigr)(\eps y)^\perp
 =\eps\bigl(\delta y^\perp-u_P(y)\bigr),
\end{equation}
which follows from $u_\D(x)=x^\perp/2$ in $\D$ and
$u_{\eps P}(\eps y)=\eps u_P(y)$.  Its right-hand side is tangent to each
component of $\pa P$, so all inner boundary equations are solved exactly
before the outer circle is perturbed.

The $D_N$ symmetry eliminates the moments of orders $1,\ldots,N-1$.
Accordingly, the first nonradial exterior defect has size
$t_{\rm o}=\eps^{N+2}$.  The outer correction is harmonic near the shrinking
polygon, and its first nonconstant invariant term has degree $N$; at
$x=\eps y$ it therefore enters the normalized inner equation at the second
scale $t_{\rm i}=\eps^{2N}$.  These a priori distinct weights coincide when
$N=2$ and differ when $N\ge3$.  Keeping them separate gives the
quantitative conclusions of Proposition~\ref{prop:01}.

After the two renormalizations, the limiting equation is affine and its
derivative is lower triangular:
\begin{equation}\label{eq:138}
 \begin{pmatrix}
   L_\delta&0\\ C_N&\mathscr A_P
 \end{pmatrix}
 \begin{pmatrix}g\\(h,\lambda)\end{pmatrix}
 +\begin{pmatrix}P_N(0)\\0\end{pmatrix}=0.
\end{equation}
The nonresonant Rankine operator $L_\delta$ and the augmented template
operator $\mathscr A_P$, obtained from Garc\'ia's reduced nondegeneracy
together with the angular-velocity mode completion in
Section~\ref{sec:02}, are isomorphisms.  The parameter-dependent
implicit-function theorem therefore applies.  In contrast
with the multilevel branch of \cite{BCGHM26}, the resulting gluing
step is a local isomorphism problem in which the vorticity vanishes in every
hole from the outset.

\textbf{Organization of the paper.}
The remainder of the paper is structured as follows.
Section~\ref{sec:01} introduces the boundary charts and the coupled
equations. Section~\ref{sec:02} constructs the fixed unit-vorticity
template, Section~\ref{sec:03} presents the exact complement identity
and the two scales, and Section~\ref{sec:04} derives the renormalized
expansions. Section~\ref{sec:gluing} solves the limiting coupled system,
while Section~\ref{sec:05} establishes the geometric properties and
quantitative asymptotics.
\endgroup

\section{\texorpdfstring{Boundary equations and functional setting}
{Boundary equations and functional setting}}
\label{sec:01}

\subsection{\texorpdfstring{Rotating-boundary criterion and scaling laws}
{Rotating-boundary criterion and scaling laws}}

We first record the geometric boundary equation and the two scaling laws
used throughout the paper.  For a bounded set $E$, let
$u_E=\nabla^\perp\psi_E$.  If the boundary of $D$ is a finite disjoint union
of $C^1$ Jordan curves, then the normal-velocity condition
\eqref{eq:002} is equivalent to
\begin{equation}\label{eq:019}
 \psi_D(x)-\frac{\Om}{2}|x|^2
 \quad\text{is constant on each connected component of }\pa D.
\end{equation}
Indeed, the relative velocity is
$\nabla^\perp(\psi_D-\Om|x|^2/2)$, so its normal component vanishes precisely
when the tangential derivative of the relative stream function vanishes on
each boundary component.  The constants on different components need not
agree.

Throughout the paper we use this tangentially differentiated form of the
boundary equation \eqref{eq:002}.  It removes the independent constants on the boundary
components and also annihilates the logarithmic constant generated when the
small holes are rescaled.

We shall also use, for $\kappa>0$ and $s>0$, the identities
\begin{equation}\label{eq:020}
 u_{\kappa\one_E}=\kappa u_E,
 \qquad
 u_{sE}(sx)=s u_E(x).
\end{equation}
The first identity in \eqref{eq:020} is the linearity of the
Biot--Savart law.  The second follows by the change of variables $y=s\xi$ in
the Biot--Savart integral. 

These two scaling operations play distinct roles in the construction.
The desingularized seed initially has vorticity amplitude
$(\pi\rho_*^2)^{-1}$; lowering that amplitude to one multiplies its angular
velocity by $\pi\rho_*^2$ without changing its boundary.  The subsequent
embedding $P\mapsto\eps P$ changes its size but leaves this unit-vorticity
angular velocity unchanged.

\subsection{\texorpdfstring{Outer and inner boundary charts}
{Outer and inner boundary charts}}

Write $w=e^{\mathrm{i}\theta}\in\T$.
The outer boundary is represented by the area-preserving polar chart
\begin{equation}\label{eq:021}
  z_r(\theta)=R_r(\theta)e^{\mathrm{i}\theta},
  \qquad
  R_r(\theta)=\sqrt{1+2r(\theta)}.
\end{equation}
Following the notation of \cite{HXX26a}, define
\begin{equation}\label{eq:Xm}
  X_N=
  \left\{
  r\in C^{2-\alpha}(\T;\R):
  r(\theta)=\sum_{k\ge1}r_{kN}\cos(kN\theta),\quad r_{kN}\in\R
  \right\}.
\end{equation}
The missing zero mode fixes the outer area exactly.  Indeed, the domain enclosed
by \eqref{eq:021} satisfies
\begin{equation}\label{eq:022}
  |D_0(r)|=\frac12\int_0^{2\pi}R_r(\theta)^2\,\mathrm{d}\theta
  =\pi+\int_0^{2\pi}r(\theta)\,\mathrm{d}\theta=\pi.
\end{equation}
The corresponding target space is
\begin{equation}\label{eq:Ym}
  Y_N=
  \left\{
  h\in C^{1-\alpha}(\T;\R):
  h(\theta)=\sum_{k\ge1}h_{kN}\sin(kN\theta),\quad h_{kN}\in\R
  \right\}.
\end{equation}
The restriction to these $D_N$-symmetric spaces makes the remaining inner
boundary equations rotated copies of the fundamental one and forces the
interior harmonic feedback to start at degree $N$.  The missing zero mode
fixes the outer area, while the cosine--sine choice fixes the reflection phase
of the domain and residual.

The inner template will be fixed in Section~\ref{sec:02}.  We record here
the chart used near it.  One fundamental component is parametrized by
the positively oriented boundary trace of a normalized exterior
conformal map
$p_0\in C^{2-\alpha}(\T;\C)$, symmetric with respect to the real axis.  The
template supplied by Proposition~\ref{prop:02} has
\[
 p_0(w)=\ell+\rho_*\bigl(w+\rho_*h_{\rho_*}(w)\bigr),
\]
where $\rho_{\ast}$ is small and $h_{\rho_*}$ belongs to the Laurent class $X_P$
defined in \eqref{eq:XP} below.
Thus $p_0(\T)$ is a small, nearly circular component with conformal center
$\ell$ and conformal radius $\rho_*$.  The subscript in $p_0$ denotes the base
point $q=0$ of the fixed chart, not the singular limit $\rho=0$.
Nearby components are parametrized by
\begin{equation}\label{eq:023}
  p_q(w)=p_0(w)+\rho_*^2q(w),
  \qquad
  q(w)=\sum_{n\ge1}a_n w^{-n},
  \qquad a_n\in\R,
\end{equation}
with $q$ small in $C^{2-\alpha}$.  The factor $\rho_*^2$ agrees with the
normalized chart in the desingularization theorem and remains fixed throughout
the gluing argument.  The absence of a constant term and of a change
in the leading conformal coefficient fixes the conformal center and conformal
radius.
\begingroup
The Laurent series in \eqref{eq:023} extend holomorphically to
$|w|>1$ and satisfy
\[
 p_q(w)=\rho_*w+O(1)\qquad (w\to\infty),
 \qquad \rho_*>0.
\]
After decreasing the $X_P$-neighborhood, $p_q'=\frac{\mathrm{d} p_q}{\mathrm{d} w}$ remains nonzero on
$|w|\ge1$, and the boundary restriction $p_q|_{\T}$ is a homeomorphism onto
a Jordan curve $\Gamma_q$.  For $z\notin\Gamma_q$, apply the argument
principle on
\[
 A_{\eta,R}:=\{1+\eta<|w|<R\}.
\]
The outer boundary of this annulus is oriented counterclockwise and its inner
boundary clockwise.  Hence, if both circles are parametrized
counterclockwise, the zero count is the winding number on $|w|=R$ minus that
on $|w|=1+\eta$.  The first winding number is one for all sufficiently large
$R$; as $\eta\downarrow0$, the second tends to one for
$z\in\operatorname{Int}(\Gamma_q)$ and to zero for
$z\in\operatorname{Ext}(\Gamma_q)$.  Letting first $R\to\infty$ and then
$\eta\downarrow0$ gives, with zeros counted with multiplicity,
\[
 N_{\{|w|>1\}}(p_q-z)
 =\begin{cases}
  1,&z\in\operatorname{Ext}(\Gamma_q),\\
  0,&z\in\operatorname{Int}(\Gamma_q).
 \end{cases}
\]
Consequently $p_q$ maps $\{|w|>1\}$ biholomorphically onto the unbounded
component $\operatorname{Ext}(\Gamma_q)$.  Thus the terms ``conformal
center'' and ``conformal radius'' above refer to the actual normalized
exterior conformal map, not only to a formal Laurent expansion.
\endgroup
The other components are
\begin{equation}\label{eq:024}
  p_{q,j}(w)=\zeta^{j-1}p_q(w),
  \qquad 1\le j\le N.
\end{equation}
For the $j$-th component, precomposing $p_{q,j}$ with
$w\mapsto\zeta^{-(j-1)}w$ restores a positive leading coefficient.  Its
normalized exterior conformal center and radius are therefore
$\ell\zeta^{j-1}$ and $\rho_*$, respectively.
We denote their interiors by $P_j(q)$ and put
\begin{equation}\label{eq:Pq}
  P(q)=\bigcup_{j=1}^{N}P_j(q).
\end{equation}
The inner domain and target spaces are
\begin{align}
  X_P&=
  \left\{
  q\in C^{2-\alpha}(\T;\C):
  q(w)=\sum_{n\ge1}a_nw^{-n},\ a_n\in\R
  \right\},\label{eq:XP}\\
  Y_P&=
  \left\{
  h\in C^{1-\alpha}(\T;\R):
  h(e^{\mathrm{i}\theta})=\sum_{n\ge1}b_n\sin(n\theta),\ b_n\in\R
  \right\}.
  \label{eq:YP}
\end{align}
As usual, these Fourier descriptions are understood as closed symmetry subspaces
of the indicated H\"older spaces.

\subsection{\texorpdfstring{Symmetry, area, and complex moments}
{Symmetry, area, and complex moments}}

We use the following complex form of the Cauchy--Green formula.  If $E$ is a
bounded $C^1$ domain, $\pa E$ is positively oriented, and
$\Phi\in C^1(\overline E)$, then
\begin{equation}\label{eq:025}
 \pa_{\bar z}:=\frac12(\pa_{x_1}+\mathrm{i}\pa_{x_2}),
 \qquad
 \int_E\pa_{\bar z}\Phi(z)\,\mathrm{d}A(z)
   =\frac1{2\mathrm{i}}\int_{\pa E}\Phi(z)\,\mathrm{d}z.
\end{equation}

\begingroup
\begin{lemma}
\label{lem:02}
On every fixed sufficiently small shape ball in $X_P$, the area map
$\mathscr A(q):=|P(q)|$ and, for every $n\ge1$, the moment map
\begin{equation}\label{eq:026}
 M_n(q):=\int_{P(q)}z^n\,\mathrm{d}A(z)
\end{equation}
are Fr\'echet $C^1$.  They satisfy
\begin{equation}\label{eq:027}
 M_n(q)=0\quad\text{if }N\nmid n,
 \qquad M_{kN}(q)\in\R.
\end{equation}
Moreover, there are $R_P>1$ and $C>0$, independent of $n$ and of $q$ in
the ball, such that
\begin{equation}\label{eq:028}
 \begin{aligned}
 |\mathscr A(q)|+\|D_q\mathscr A(q)\|_{\mathcal L(X_P,\C)}&\le C,\\
 |M_n(q)|&\le C R_P^n,\\
 \|D_qM_n(q)\|_{\mathcal L(X_P,\C)}&\le C(1+n)R_P^n.
 \end{aligned}
\end{equation}
\end{lemma}
\endgroup

\begin{proof}
\begingroup
Applying \eqref{eq:025} with $\Phi(z)=\overline z$ on each
positively oriented component gives
\begin{equation}\label{eq:029}
 \mathscr A(q)
 =\frac1{2\mathrm{i}}\sum_{j=1}^N\int_{\T}
 \overline{p_{q,j}(w)}\,p_{q,j}'(w)\,\mathrm{d}w.
\end{equation}
If $p_j=p_{q,j}$ and
$\dot p_j=\rho_*^2\zeta^{j-1}\dot q$, differentiation yields
\begin{equation}\label{eq:030}
 D_q\mathscr A(q)[\dot q]
 =\frac1{2\mathrm{i}}\sum_{j=1}^N\int_{\T}
 \bigl(\overline{\dot p_j}\,p_j'
       +\overline{p_j}\,\dot p_j'\bigr)\,\mathrm{d}w.
\end{equation}
Thus $\mathscr A$ is $C^1$ with the asserted uniform bound.

Rotation by $\zeta$ gives $M_n(q)=\zeta^nM_n(q)$, and reflection gives
$M_{kN}(q)=\overline{M_{kN}(q)}$, proving
\eqref{eq:027}.  Applying
\eqref{eq:025} with $\Phi(z)=z^n\overline z$ gives
\begin{equation}\label{eq:031}
 M_n(q)=\frac1{2\mathrm{i}}\sum_{j=1}^N
 \int_{\pa P_j(q)}z^n\overline z\,\mathrm{d}z.
\end{equation}
This boundary formula shows that each moment map is $C^1$ and avoids any
interior extension of the chart.  Its derivative is
\begin{align}\label{eq:032}
 D_qM_n(q)[\dot q]
 =\frac1{2\mathrm{i}}\sum_{j=1}^N\int_{\T}
 \bigl(n p_j^{n-1}\dot p_j\overline{p_j}p_j'
       +p_j^n\overline{\dot p_j}p_j'+p_j^n\overline{p_j}\dot p_j'\bigr)\,\mathrm{d}w.
\end{align}
On a fixed small shape ball the charts and their first derivatives are
uniformly bounded, while the component boundaries remain embedded and
separated.  Choosing $R_P>1$ with $P(q)\Subset B_{R_P}$ in that ball and
estimating the last two displays proves \eqref{eq:028}.
\endgroup
\end{proof}

\begingroup
In particular, after tangential differentiation on the circular reference
boundary, the first nonconstant multipole in the outer residual is the
$N$-th one and occurs at scale $\eps^{N+2}$.  The monopole is constant on
the reference circle; on the perturbed boundary its contribution belongs to
the $O(\eps^{N+4})$ remainder.
\endgroup

\subsection{\texorpdfstring{The coupled boundary equations}
{The coupled boundary equations}}

For $r\in X_N$ small, let $D_0(r)$ be the domain enclosed by \eqref{eq:021}.
For $\eps>0$ define
\begin{equation}\label{eq:033}
  D(\eps,r,q)
  =D_0(r)\setminus\overline{\eps P(q)}.
\end{equation}
The vorticity is therefore exactly
\begin{equation}\label{eq:034}
  \one_{D(\eps,r,q)}
  =\one_{D_0(r)}-\one_{\eps P(q)}.
\end{equation}
In particular, the vorticity in every inner component is zero, not a second
nonzero vorticity level.

Using the relative stream function \eqref{eq:003} and the boundary
criterion \eqref{eq:019} for $D(\eps,r,q)$ on the outer
boundary $x=z_r(\theta)$ gives
\[
 \pa_\theta\left[
  \psi_{D_0(r)}(z_r(\theta))-\psi_{\eps P(q)}(z_r(\theta))
  -\frac{\Om}{2}|z_r(\theta)|^2\right]=0.
\]
Since $|z_r|^2=1+2r$, the outer boundary equation is
\begin{equation}\label{eq:F0-def}
\begin{split}
  \cF_0(\eps,\Om,r,q)(\theta)
  ={}&\pa_\theta\Big[
  \psi_{D_0(r)}(z_r(\theta))
   -\psi_{\eps P(q)}(z_r(\theta))\Big]
  -\Om r'(\theta)=0.
\end{split}
\end{equation}
The last term is obtained by differentiating
$-\Om|z_r|^2/2=-\Om(1+2r)/2$.
For the inner equation we use the rescaled variable $x=\eps y$.  Set, for
$\eps>0$,
\begin{equation}\label{eq:Qeps}
  Q_{\eps,r}(y)
  =\frac{\psi_{D_0(r)}(\eps y)-\psi_{D_0(r)}(0)}{\eps^2}.
\end{equation}
The equation on the fundamental inner component is
\begin{equation}\label{eq:F1-def}
  \cF_1(\eps,\Om,r,q)(\theta)
  =\pa_\theta\left[
  Q_{\eps,r}(p_q(e^{\mathrm{i}\theta}))
  -\psi_{P(q)}(p_q(e^{\mathrm{i}\theta}))
  -\frac{\Om}{2}|p_q(e^{\mathrm{i}\theta})|^2
  \right]=0.
\end{equation}
Indeed, the scaling identity
\begin{equation}\label{eq:035}
  \psi_{\eps P(q)}(\eps y)
  =\eps^2\psi_{P(q)}(y)
   +\frac{\eps^2|P(q)|}{2\pi}\log\eps,
\end{equation}
shows that \eqref{eq:F1-def} is precisely $\eps^{-2}$ times the tangential
derivative of the physical relative stream function.  The last term in
\eqref{eq:035} is constant on every inner boundary and disappears
after tangential differentiation.

By $D_N$-symmetry, the equations on the other $N-1$ inner components are
rotated copies of \eqref{eq:F1-def}.  Thus the full rotating-boundary system is
defined, for suitable small open neighborhoods
$\mathcal U_N\subset X_N$ and $\mathcal U_P\subset X_P$, by
\begin{equation}\label{eq:036}
\begin{aligned}
 \cF:\ (0,\eps_1)\times\R\times\mathcal U_N\times\mathcal U_P&\longrightarrow Y_N\times Y_P,\\
 (\eps,\Om,r,q)&\longmapsto
 \bigl(\cF_0(\eps,\Om,r,q),\cF_1(\eps,\Om,r,q)\bigr),\\
 \cF(\eps,\Om,r,q)&=0.
\end{aligned}
\end{equation}
Here $\eps_1>0$ and the two neighborhoods are chosen so that the boundary
charts are embedded and separated and
$\eps\overline{P(q)}\Subset D_0(r)$.  No restriction on $\Om\in\R$ is needed
at this level; the dependence on $\Om$ is affine.

In \eqref{eq:036}, we place $\Om$ after $\eps$ to retain the
customary notation $F(\Om,r)$ for $V$-state equations.  In every
linearized operator below, however, the unknowns are ordered as $(r,q,\Om)$.
Thus $D_{(r,q,\Om)}\cF$ denotes the derivative in the last three variables
after this fixed permutation of factors.

\section{\texorpdfstring{A fixed nonresonant co-rotating template}
{A fixed nonresonant co-rotating template}}
\label{sec:02}

We begin by recasting Garc\'ia's polygonal branch in the normalization used below.

\begin{lemma}
\label{lem:03}
Fix $N\ge2$, $\ell>0$, and $\alpha\in(0,1)$.  There exist
$0<\rho_1\le\rho_0$ and $C^1$ maps
\[
 [0,\rho_0)\ni\rho\longmapsto h_\rho\in X_P,
 \qquad
 [0,\rho_0)\ni\rho\longmapsto\delta_\rho\in\R,
\]
with the following properties.

\begin{enumerate}[label=\textnormal{(\roman*)},leftmargin=2.4em]
\item One has $h_0=0$, $\delta_0=0$, and
\begin{equation}\label{eq:037}
 \|h_\rho\|_{X_P}=O(\rho).
\end{equation}
For $0<\rho<\rho_0$, let
\begin{equation}\label{eq:038}
 p_{\rho,h}(w)=\ell+\rho\bigl(w+\rho h(w)\bigr),
 \qquad
 P_\rho(h)=\bigcup_{j=1}^N\zeta^{j-1}P_{\rho,1}(h),
\end{equation}
where $p_{\rho,h}$ parametrizes $\partial P_{\rho,1}(h)$.  After decreasing
$\rho_0$, the set $P_\rho(h_\rho)$ is $D_N$-invariant and its components are
pairwise disjoint simply connected domains with $C^{2-\alpha}$ boundaries,
and the vorticity $\one_{P_\rho(h_\rho)}$ is a co-rotating configuration of
$N$ unit-vorticity vortex patches with angular velocity
\begin{equation}\label{eq:039}
 \delta_\rho
 =\frac{N-1}{4\ell^2}\rho^2+O(\rho^3),
\end{equation}
and
\begin{equation}\label{eq:040}
 \delta_\rho'
 =\frac{N-1}{2\ell^2}\rho+O(\rho^2).
\end{equation}

\item
The map
\begin{equation}\label{eq:041}
 \cG_\rho^{\rm tan}(\mu,h)
 :=\pa_\theta\left[
 \psi_{P_\rho(h)}(p_{\rho,h}(e^{\mathrm{i}\theta}))
 -\frac\mu2|p_{\rho,h}(e^{\mathrm{i}\theta})|^2
 \right],
\end{equation}
is Fr\'echet $C^1$ in a neighborhood of
$(\delta_\rho,h_\rho)$ and one has
\begin{equation}\label{eq:042}
 \cG_\rho^{\rm tan}(\delta_\rho,h_\rho)=0,
 \qquad
 D_{(\mu,h)}\cG_\rho^{\rm tan}(\delta_\rho,h_\rho):
 \R\times X_P\longrightarrow Y_P
\end{equation}
as a bounded linear isomorphism for every $\rho\in(0,\rho_1)$.
\end{enumerate}
\end{lemma}

\begin{proof}
\noindent\emph{External input and space identification.}
For $f\in Y_P$, define
\begin{equation}\label{eq:043}
 \pi_1f:=\frac1\pi\int_0^{2\pi}f(\theta)\sin\theta\,\mathrm{d}\theta,
 \qquad
 \Pi_{\ge2}f:=f-\pi_1(f)\sin\theta.
\end{equation}
Integration against $\sin\theta$ is a bounded functional on
$C^{1-\alpha}(\T)$, so
\[
 \pi_1\in\mathcal L(Y_P,\R),
 \qquad
 \Pi_{\ge2}\in\mathcal L(Y_P).
\]
The second operator is a bounded projection with closed range and kernel
$\operatorname{span}\{\sin\theta\}$.  Hence
\begin{equation}\label{eq:044}
 Y_P=\operatorname{span}\{\sin\theta\}
 \oplus\Pi_{\ge2}Y_P
\end{equation}
is a topological direct sum.  Set $\gamma=1-\alpha$.  Under the
identifications $X_{1+\gamma}=X_P$ and
$Y_\gamma=\Pi_{\ge2}Y_P$, Garc\'ia's construction provides the reduced
$C^1$ branch, a selected high-intensity angular velocity
$\widehat\delta_\rho$, and a regularized full residual
\cite[(9)--(10), (28), (34), (43), (48), Propositions~3.4 and~3.7,
and Theorem~3.8]{Gar21}.  The branch satisfies
\[
 h_0=0,
 \qquad
 \widehat\delta_0=\frac{N-1}{4\pi\ell^2}.
\]
The implicit-function theorem defining that branch gives
$\|h_\rho\|_{X_P}=O(\rho)$ and
$\widehat\delta_\rho=\widehat\delta_0+O(\rho)$.  This proves
\eqref{eq:037}.  For $0<\rho<\rho_0$, the corresponding
physical reconstruction is
\begin{equation}\label{eq:045}
 \widehat\omega_\rho
 =\frac1{\pi\rho^2}\one_{P_\rho(h_\rho)},
\end{equation}
which rotates with angular velocity $\widehat\delta_\rho$ and has the
geometry asserted in \textnormal{(i)}.
\medskip
\noindent\emph{Restoration of the first sine mode.}
To recover the full target $Y_P$, put
$A_\rho=(\pi\rho^2)^{-1}$,
$N_{\rho,h}(w)=w(1+\rho h'(w))$, and define
\begin{equation}\label{eq:046}
 \widehat{\mathscr F}_\rho(\widehat\delta,h)
 :=\left[A_\rho u_{P_\rho(h)}(p_{\rho,h})
       -\widehat\delta\,p_{\rho,h}^{\perp}\right]\cdot N_{\rho,h}.
\end{equation}
If
$\Phi_{\rho,h}(w)=\rho(w+\rho h(w))$, then
$p_{\rho,h}=\ell+\Phi_{\rho,h}$ and
$w\Phi_{\rho,h}'(w)=\rho N_{\rho,h}(w)$.  Identifying planar vectors with
complex numbers and using $a\cdot b=\Rea(a\overline b)$, one has the exact
normalization identity
\begin{align}
 &\left[A_\rho u_{P_\rho(h)}(p_{\rho,h})
       -\widehat\delta\,p_{\rho,h}^{\perp}\right]\cdot N_{\rho,h}\notag\\
 &\qquad\qquad=\Rea\!\left[
 \left\{\overline{A_\rho u_{P_\rho(h)}(p_{\rho,h})}
       +\mathrm{i}\widehat\delta\,\overline{p_{\rho,h}}\right\}
 N_{\rho,h}\right].
 \label{eq:047}
\end{align}
Garc\'ia's contour equation \cite[(28)]{Gar21} is the right-hand side
of \eqref{eq:047} multiplied by $\rho$, because its
tangent factor is $w\Phi_{\rho,h}'=\rho N_{\rho,h}$.  Division by this
nonzero scalar therefore gives exactly
\eqref{eq:046}, with no sign or conjugation change.

We next make the regularity of the first-mode selector explicit.  Let
$\mathcal J(\rho,h)$ denote Garc\'ia's regularized complex velocity
$J(\rho,h)$ in \cite[(33)]{Gar21}, after the identification
$(\epsilon,f)=(\rho,h)$, and set
\begin{align}
 \mathcal N(\rho,h)
 &:=\mathrm{i}\int_{\T}\mathcal J(\rho,h)(w)(w-\bar w)
       (1+\rho h'(w))\,\mathrm{d}w,
 \label{eq:048}\\
 \mathcal D(\rho,h)
 &:=\int_{\T}(1+\rho h'(w))(w-\bar w)
       \overline{p_{\rho,h}(w)}\,\mathrm{d}w\notag\\
 &=\int_{\T}(1+\rho h'(w))(w-\bar w)
       \bigl(\ell+\rho\bar w+\rho^2h(\bar w)\bigr)\,\mathrm{d}w.
 \label{eq:049}
\end{align}
Here the last equality uses the real Laurent coefficients of $h$.  In the
normalization verified in \eqref{eq:047}, Garc\'ia's
formula \cite[(43)]{Gar21} reads exactly
\begin{equation}\label{eq:050}
 \Lambda(\rho,h)=\frac{\mathcal N(\rho,h)}{\mathcal D(\rho,h)}.
\end{equation}
The divided-difference regularization in \cite[(33)]{Gar21}, with the
estimates used in the proof of \cite[Proposition~3.7]{Gar21}, shows that
$(\rho,h)\mapsto\mathcal J(\rho,h)$, and hence $\mathcal N$, is Fr\'echet
$C^1$ on a fixed neighborhood of $(0,0)$.  The denominator $\mathcal D$ is
Fr\'echet $C^1$ directly from \eqref{eq:049}; when $\rho=0$,
\begin{equation}\label{eq:051}
 \mathcal D(0,h)
 =\ell\int_{\T}(w-\bar w)\,\mathrm{d}w
 =-2\pi\mathrm{i}\ell\ne0.
\end{equation}
After shrinking that fixed neighborhood, $\mathcal D$ stays nonzero, and
\eqref{eq:050} proves $\Lambda\in C^1$.  Proposition~3.4 of \cite{Gar21} gives
\begin{equation}\label{eq:052}
 \pi_1\widehat{\mathscr F}_\rho(\Lambda(\rho,h),h)=0,
 \qquad
 \Lambda(0,h)=\widehat\delta_0.
\end{equation}
Since the map is constant in $h$, it follows that
$D_h\Lambda(0,h)=0$, in particular $D_h\Lambda(0,0)=0$.  Set
\[
 \widetilde{\mathscr F}(\rho,h)
 :=\Pi_{\ge2}\widehat{\mathscr F}_\rho(\Lambda(\rho,h),h).
\]
The residual already belongs to $Y_P$; hence elimination of its first sine
coefficient implies
$\widehat{\mathscr F}_\rho(\Lambda(\rho,h),h)
=\widetilde{\mathscr F}(\rho,h)$.  Since the angular velocity enters
affinely, the exact full--reduced identity
\begin{equation}\label{eq:053}
 \widehat{\mathscr F}_\rho(\widehat\delta,h)
 =\widetilde{\mathscr F}(\rho,h)
  +(\widehat\delta-\Lambda(\rho,h))\mathscr B(\rho,h),
 \qquad
 \mathscr B=D_{\widehat\delta}\widehat{\mathscr F}_\rho,
\end{equation}
holds in the full space $Y_P$.  The cited reduced map
$\widetilde{\mathscr F}$ and the selector $\Lambda$ have $C^1$ extensions
to $\rho=0$.  Moreover,
\[
 \mathscr B(\rho,h)
 =-p_{\rho,h}^{\perp}\cdot N_{\rho,h}
\]
is explicit and $C^\infty$ in $(\rho,h)$, including at $\rho=0$.
Consequently the right-hand side of \eqref{eq:053} is a
$C^1$ map on an open neighborhood of
$(\rho,\widehat{\delta},h)=(0,\widehat\delta_0,0)$ in $\R\times\R\times X_P$, now with values in the
full target $Y_P$.

More precisely, there is $r_0>0$ such that
\begin{equation}\label{eq:054}
 \Lambda(0,h)=\widehat\delta_0,
 \qquad D_h\Lambda(0,h)=0
 \quad\text{for }\|h\|_{X_P}<r_0,
 \qquad
 \mathscr B(0,h)=-\ell\sin\theta.
\end{equation}
The last identity follows directly from
$\mathscr B(\rho,h)=-p_{\rho,h}^\perp\cdot N_{\rho,h}$.  The
derivative in the cited construction is
therefore
\begin{equation}\label{eq:055}
\begin{aligned}
 D_{(\widehat\delta,h)}
 \widehat{\mathscr F}_0(\widehat\delta_0,0)[\lambda,k]
 &=-\ell\lambda\sin\theta
   -\frac1{2\pi}\Rea\!\left(
      \mathrm{i}k'(e^{\mathrm{i}\theta})\right)\\
 &=-\ell\lambda\sin\theta
   +\frac1{2\pi}\sum_{n\ge1}nk_n\sin((n+1)\theta),
\end{aligned}
\end{equation}
where $k(w)=\sum_{n\ge1}k_nw^{-n}$.  Its first-mode coefficient in the
angular-velocity direction is $-\ell\ne0$.  More explicitly, for
\[
 y(\theta)=\sum_{m\ge1}b_m\sin(m\theta)\in Y_P,
\]
the unique inverse image under \eqref{eq:055} is
\begin{equation}\label{eq:056}
 \lambda=-\frac{b_1}{\ell},
 \qquad
 k(w)=\sum_{n\ge1}\frac{2\pi b_{n+1}}{n}w^{-n}.
\end{equation}
The periodic H\"older multiplier estimate, together with the
one-mode index shift, gives
\[
 |\lambda|+\|k\|_{X_P}\le C\|y\|_{Y_P}.
\]
Thus \eqref{eq:055} is a bounded linear isomorphism
from $\R\times X_P$ onto $Y_P$.  The $C^1$ extension and the convergence
$(\rho,\widehat\delta_\rho,h_\rho)\to(0,\widehat\delta_0,0)$ imply convergence
of the corresponding derivatives in operator norm.  Openness of the set of
isomorphisms then gives, after decreasing $\rho_1$,
\begin{equation}\label{eq:057}
 D_{(\widehat\delta,h)}
 \widehat{\mathscr F}_\rho(\widehat\delta_\rho,h_\rho):
 \R\times X_P\longrightarrow Y_P
\end{equation}
as a bounded linear isomorphism for $0<\rho<\rho_1$.

\medskip
\noindent\emph{Unit-vorticity normalization.}
Finally, the amplitude-scaling law \eqref{eq:020} gives
$\delta_\rho=\pi\rho^2\widehat\delta_\rho$.  Since the branch is $C^1$,
$\widehat\delta_\rho'=O(1)$, and hence
\[
 \delta_\rho'
 =2\pi\rho\widehat\delta_\rho
  +\pi\rho^2\widehat\delta_\rho',
\]
which proves \eqref{eq:039}--\eqref{eq:040}.  Moreover,
$\pa_\theta p_{\rho,h}=\mathrm{i}\rho N_{\rho,h}$, so the streamline
criterion and \eqref{eq:046} give the exact conjugacy
\begin{equation}\label{eq:058}
 \cG_\rho^{\rm tan}(\mu,h)
 =-\pi\rho^3\widehat{\mathscr F}_\rho
   \left(\frac\mu{\pi\rho^2},h\right).
\end{equation}
It follows first that
$\cG_\rho^{\rm tan}(\delta_\rho,h_\rho)=0$ and, after differentiation, that
\begin{equation}\label{eq:059}
 D_{(\mu,h)}\cG_\rho^{\rm tan}(\delta_\rho,h_\rho)
 =-\pi\rho^3
 D_{(\widehat\delta,h)}\widehat{\mathscr F}_\rho
  (\widehat\delta_\rho,h_\rho)
 \begin{pmatrix}(\pi\rho^2)^{-1}&0\\0&I\end{pmatrix}.
\end{equation}
For every fixed $\rho>0$, both outer factors are invertible.  Part
\textnormal{(ii)} follows from \eqref{eq:057}. The proof is thus complete.
\end{proof}

We next select from this branch a fixed nonresonant template for the
gluing construction.
\begin{proposition}\label{prop:02}
For every integer $N\ge2$, every $\ell>0$, and every $\alpha\in(0,1)$, there
exist $\rho_2>0$, a discrete set
$\mathcal E_N\subset(0,\rho_2)$ whose only possible accumulation point in
$[0,\rho_2)$ is $0$, and a family
$(P_\rho,\delta_\rho,p_\rho)_{0<\rho<\rho_2}$ with the following properties.

\begin{enumerate}[label=\textnormal{(\roman*)},leftmargin=2.4em]
\item The fundamental chart is
\[
 p_\rho(w)=\ell+\rho\bigl(w+\rho h_\rho(w)\bigr),
\]
where $\rho\mapsto h_\rho$ is $C^1$ from $[0,\rho_2)$ into $X_P$ and
$h_0=0$.  Each $P_\rho$ is a $D_N$-invariant union of $N$ pairwise disjoint
simply connected $C^{2-\alpha}$ domains, and the vorticity $\one_{P_\rho}$
is a co-rotating configuration of $N$ unit-vorticity vortex patches with
angular velocity $\delta_\rho\in(0,1/(2N))$.

\item For every
$\rho_*\in(0,\rho_2)\setminus\mathcal E_N$, set
\[
 p_0:=p_{\rho_*},\qquad P:=P_{\rho_*},\qquad
 \delta:=\delta_{\rho_*}.
\]
Then
\begin{equation}\label{eq:060}
 \delta\ne\frac1{2kN}
 \qquad\text{for every }k\ge1.
\end{equation}

\item For the chart $p_q=p_{\rho_*}+\rho_*^2q$ and its associated
configuration $P(q)$, there exist $\eta_\mu>0$ and an open neighborhood
$\mathcal V_P\subset\mathcal U_P$ of $0$ such that
\begin{equation}\label{eq:061}
\begin{aligned}
 \cG:
 &(\delta-\eta_\mu,\delta+\eta_\mu)\times\mathcal V_P
 \longrightarrow Y_P,\\
 \cG(\mu,q)(\theta)
 &:={\pa_\theta}\left[
  \psi_{P(q)}(p_q(e^{\mathrm{i}\theta}))
  -\frac{\mu}{2}|p_q(e^{\mathrm{i}\theta})|^2
  \right]
\end{aligned}
\end{equation}
is Fr\'echet $C^1$, satisfies
\begin{equation}\label{eq:062}
 \cG(\delta,0)=0,
\end{equation}
and the operator
\begin{equation}\label{eq:063}
\begin{aligned}
 \mathscr A_P:X_P\times\R&\longrightarrow Y_P,\\
 \mathscr A_P[h,\lambda]
 &:=-D_q\cG(\delta,0)h+D_\mu\cG(\delta,0)\lambda
\end{aligned}
\end{equation}
is a bounded linear isomorphism.

\item As $\rho\downarrow0$,
\begin{equation}\label{eq:064}
 \delta_\rho
 =\frac{N-1}{4\ell^2}\rho^2+O(\rho^3),
 \qquad
 M_N(P_\rho)=N\pi\ell^N\rho^2+O(\rho^4).
\end{equation}
Moreover,
\[
 M_N(P_\rho)>0,
 \qquad 0<\rho<\rho_2.
\]
\end{enumerate}
After one admissible value $\rho_*$ has been selected, the objects
$\rho_*$, $P$, $\delta$, and $p_0$ are kept fixed; the embedding parameter
$\eps$ introduced later is independent of this choice.
\end{proposition}
\begin{proof}
Let the branch
\[
 \rho\longmapsto
 \bigl(P_\rho(h_\rho),\widehat\delta_\rho,\delta_\rho\bigr),
 \qquad 0\le\rho<\rho_0,
\]
be supplied by Lemma~\ref{lem:03}.  We write
\(P_\rho:=P_\rho(h_\rho)\).

\medskip
\noindent\emph{Step 1. The first nonradial moment.}
Let \(Q_\rho\) be the domain bounded by
\[
 q_\rho(w):=w+\rho h_\rho(w).
\]
On \(\T\), write
\[
 \widetilde h_\rho(w):=\overline{h_\rho(w)}
 =\sum_{n\ge1}a_n(\rho)w^n.
\]
Applying \eqref{eq:025} with
\(\Phi(z)=\overline z\) gives
\[
 |Q_\rho|
 =\pi\left(1-\rho^2\sum_{n\ge1}n|a_n(\rho)|^2\right)
 =\pi+O(\rho^4),
\]
because \(\|h_\rho\|_{X_P}=O(\rho)\) and
\(X_P\hookrightarrow H^{1/2}(\T)\).  Applying the same formula with
\(\Phi(z)=z\overline z\), we obtain
\[
 \int_{Q_\rho}z\,\mathrm{d}A(z)
 =\frac1{2\mathrm{i}}\int_{\T}
 q_\rho(w)\overline{q_\rho(w)}q_\rho'(w)\,\mathrm{d}w.
\]
The integrand expands as
\begin{align*}
q_\rho\overline{q_\rho}q_\rho'
={}&1+\rho\bigl(w\widetilde h_\rho
       +w^{-1}h_\rho+h_\rho'\bigr)\\
&+\rho^2\bigl[h_\rho\widetilde h_\rho
 +(w\widetilde h_\rho+w^{-1}h_\rho)h_\rho'\bigr]
 +\rho^3h_\rho\widetilde h_\rho h_\rho'.
\end{align*}
The coefficient of \(w^{-1}\) in the term linear in \(\rho\) vanishes:
\(w\widetilde h_\rho\) contains only positive powers, whereas
\(w^{-1}h_\rho\) and \(h_\rho'\) contain only powers \(w^{-n-1}\),
\(n\ge1\).  Since
\[
 \frac1{2\mathrm{i}}\int_\T F(w)\,\mathrm{d}w
 =\pi\times \text{ the coefficient of }F \text{ at }w^{-1},
\]
only this Laurent coefficient contributes.  It follows that
\[
 \left|\int_{Q_\rho}z\,\mathrm{d}A(z)\right|
 \le C\rho^2\left(\|h_\rho\|_{C^1}^2
       +\rho\|h_\rho\|_{C^1}^3\right)
 =O(\rho^4).
\]
The domains \(Q_\rho\) remain uniformly bounded; hence
\begin{equation}\label{eq:065}
\begin{aligned}
 |Q_\rho|=\pi+O(\rho^4),\qquad
 \int_{Q_\rho}z\,\mathrm{d}A(z)=O(\rho^4),\\
 \int_{Q_\rho}z^k\,\mathrm{d}A(z)=O_k(1)
 \qquad\qquad \text{for each fixed }k\ge1.
\end{aligned}
\end{equation}
Since \((\zeta^jz)^N=z^N\), a change of variables on the fundamental
component followed by the binomial formula yields
\begin{align}
 M_N(P_\rho)
 &=N\rho^2\int_{Q_\rho}(\ell+\rho z)^N\,\mathrm{d}A(z)\notag\\
 &=N\rho^2\sum_{k=0}^N\binom Nk\ell^{N-k}\rho^k
   \int_{Q_\rho}z^k\,\mathrm{d}A(z)\notag\\
 &=N\pi\ell^N\rho^2+O(\rho^4).
 \label{eq:seed-MN}
\end{align}
Indeed, the terms \(k=0\), \(k=1\), and \(k\ge2\) contribute, respectively,
\(N\pi\ell^N\rho^2+O(\rho^6)\), \(O(\rho^7)\), and \(O(\rho^4)\).
Reflection symmetry makes \(M_N(P_\rho)\) real.  Since \(\ell>0\), there is
\(\rho_M>0\) such that
\[
 M_N(P_\rho)>0\qquad\text{for }0<\rho<\rho_M.
\]

\medskip
\noindent\emph{Step 2. Selection of a nonresonant microscopic scale.}
By \eqref{eq:039}--\eqref{eq:040}, after decreasing
\(\rho_2>0\) so that
\(\rho_2\le\min\{\rho_1,\rho_M\}\), we have
\begin{equation}\label{eq:066}
 0<\delta_\rho<\frac1{2N},
 \qquad \delta_\rho'>0,
 \qquad M_N(P_\rho)>0
 \qquad (0<\rho<\rho_2).
\end{equation}
Define
\[
 \mathcal E_N
 :=\left\{\rho\in(0,\rho_2):
 \delta_\rho=\frac1{2kN}\text{ for some }k\ge1\right\}.
\]
Strict monotonicity shows that each resonant value has at most one preimage.
If distinct points of \(\mathcal E_N\) accumulated at
\(\bar\rho\in(0,\rho_2)\), the corresponding numbers \(1/(2kN)\) would
accumulate at \(\delta_{\bar\rho}>0\), which is impossible.  Thus
\(\mathcal E_N\) is discrete in \((0,\rho_2)\), and its only possible
accumulation point in \([0,\rho_2)\) is \(0\).  Since it is countable, its
complement meets every interval \((0,\eta)\), \(\eta>0\).  Choose
\begin{equation}\label{eq:067}
 \rho_*\in(0,\rho_2)\setminus\mathcal E_N.
\end{equation}
Then \eqref{eq:060} holds.

\medskip
\noindent\emph{Step 3. The fixed normalized chart.}
Set
\begin{equation}\label{eq:068}
 p_0(w)=\ell+\rho_*
 \bigl(w+\rho_*h_{\rho_*}(w)\bigr),
 \qquad
 P=P_{\rho_*},
 \qquad
 \delta=\delta_{\rho_*}.
\end{equation}
Lemma~\ref{lem:03}\textnormal{(i)} gives the asserted geometry,
regularity, and rotation law.  In the fixed chart \eqref{eq:023},
\[
 p_q(w)=\ell+\rho_*
 \bigl(w+\rho_*(h_{\rho_*}(w)+q(w))\bigr),
\]
and therefore
\begin{equation}\label{eq:069}
 \cG(\mu,q)
 =\cG_{\rho_*}^{\rm tan}(\mu,h_{\rho_*}+q),
 \qquad |\mu-\delta|<\eta_\mu,\quad q\in\mathcal V_P.
\end{equation}
The root and augmented isomorphism in
Lemma~\ref{lem:03}\textnormal{(ii)} imply
\begin{equation}\label{eq:070}
 \cG(\delta,0)=0
\end{equation}
and show directly that
\[
 \mathscr A_P[h,\lambda]
 =D_{(\mu,q)}\cG(\delta,0)[\lambda,-h].
\]
The derivative on the right and the sign-changing permutation
$(h,\lambda)\mapsto(\lambda,-h)$ are bounded linear isomorphisms, which
proves \eqref{eq:063}.  The asymptotics
\eqref{eq:039} and \eqref{eq:seed-MN} give
\eqref{eq:064}.  This completes the proof.

\end{proof}

\section{\texorpdfstring{The exact complement identity and the two scales}
{The exact complement identity and the two scales}}
\label{sec:03}

Choose one nonresonant seed supplied by Proposition~\ref{prop:02} and
keep $\rho_*$, $P$, $p_0$, and $\delta$ fixed from now on.  All constants
below, including the norms of $L_\delta^{-1}$ and $\mathscr A_P^{-1}$, may
depend on this choice; the only parameter tending to zero is $\eps$.  Put
\begin{equation}\label{eq:071}
 \Om_*:=\frac12-\delta.
\end{equation}
The fixed set $P$ is the union of the component domains in a co-rotating
configuration of $N$ unit-vorticity vortex patches.  The holes will be
obtained by removing the small spatial copy $\eps P$ from the outer vortex
patch domain.

Inside the unit disk,
\begin{equation}\label{eq:072}
 \psi_\D(x)=C_\D+\frac14|x|^2,
 \qquad
 u_\D(x)=\frac12x^\perp.
\end{equation}
Since the vorticity $\one_P$ is a co-rotating configuration of
unit-vorticity vortex patches with angular velocity $\delta$,
\[
 (u_P(y)-\delta y^\perp)\cdot n_P(y)=0,
 \qquad y\in\pa P.
\]
Together with $u_{\eps P}(\eps y)=\eps u_P(y)$, this gives the exact
complement identity
\begin{equation}\label{eq:073}
 u_{\D\setminus\overline{\eps P}}(\eps y)
 -\Om_*(\eps y)^\perp
 =-\eps\bigl(u_P(y)-\delta y^\perp\bigr).
\end{equation}
The right-hand side is tangent to every component of $\pa P$.
Equivalently, for $y$ in any fixed bounded set and for all sufficiently
small $\eps$,
\begin{equation}\label{eq:074}
 \Psi_{\D\setminus\overline{\eps P},\Om_*}(\eps y)
 =C_\eps-\eps^2
 \left(\psi_P(y)-\frac\delta2|y|^2\right).
\end{equation}
Thus the unperturbed complement has no inner residual.  Its only defect is on
the outer circle.

The outer and inner corrections are governed by two a priori distinct
weights.  These weights coincide when $N=2$ and differ when $N\ge3$.  We set
\begin{equation}\label{eq:075}
 t_{\rm o}:=\eps^{N+2},
 \qquad
 t_{\rm i}:=\eps^{2N}.
\end{equation}
For $g\in X_N$, $h\in X_P$, and $\lambda\in\R$, define
\begin{align}
 r&=t_{\rm o}g,
 &q&=t_{\rm i}h,
 &\delta_{\eps,\lambda}&=\delta-t_{\rm i}\lambda,
 &\Om_{\eps,\lambda}&=\Om_*+t_{\rm i}\lambda,
 \label{eq:076}
\end{align}
so that $\Om_{\eps,\lambda}=1/2-\delta_{\eps,\lambda}$.  The corresponding
physical domain is
\begin{equation}\label{eq:077}
 D_{\eps,g,h}
 =D_0(t_{\rm o}g)\setminus\overline{\eps P(t_{\rm i}h)}.
\end{equation}

\begingroup
The first nonradial exterior moment of $P$ has order $N$, so its trace at
unit distance is $\eps^{N+2}$ and forces an outer displacement of size
$t_{\rm o}$, see \eqref{eq:078}.  The potential of that displacement is harmonic near the origin
and $D_N$-invariant.  Its first nonconstant term has degree $N$, so evaluation
at $x=\eps y$ adds the factor $\eps^N$ and produces a physical potential of
size $\eps^{2N+2}$.  The normalized inner potential in \eqref{eq:Qeps}
contains the factor $\eps^{-2}$; the feedback seen by the inner equation is
therefore
\[
 \frac{t_{\rm o}}{\eps^2}\,\eps^N
 =\eps^{2N}=t_{\rm i}.
\]
This is also the scale of an inner shape or angular-velocity perturbation.
For $N=2$, the degree-$N$ jet is the Hessian of the outer potential and gives
the usual strain block.  For $N\ge3$, rotational symmetry makes the Hessian
scalar and harmonicity makes its trace zero, so the lower-degree coupling
vanishes and the first feedback occurs at degree $N$.
\endgroup

\section{The renormalized boundary equations}
\label{sec:04}

\subsection{\texorpdfstring{Exterior multipole expansion}
{Exterior multipole expansion}}
\label{sub:01}

We begin with the exterior multipole expansion.  By
Lemma~\ref{lem:02}, $\mathscr A(q)=|P(q)|$ and all moments
used below depend Fr\'echet $C^1$ on the inner shape.

\begingroup
\begin{lemma}\label{lem:04}
Let $0<r_-<r_+<\infty$ and set
\[
 \mathcal A:=\{z\in\C:r_-\le|z|\le r_+\}.
\]
There exist $r_P>0$, $R_{\rm seed}>0$, and
$\eps_{\rm mp}\in(0,\eps_1]$ such that
\[
\begin{gathered}
 \overline B_{X_P}(0,2r_P)\subset\mathcal U_P,
 \qquad
 P(q)\Subset B_{R_{\rm seed}}\quad
 \text{for }\|q\|_{X_P}<2r_P,\\
 \frac{\eps_{\rm mp}R_{\rm seed}}{r_-}<\frac12.
\end{gathered}
\]
Moreover, for $0<\eps\le\eps_{\rm mp}$,
$q\in B_{X_P}(0,2r_P)$, and $z\in\mathcal A$,
\begin{align}
 \psi_{\eps P(q)}(z)
 ={}&\frac{\eps^2\mathscr A(q)}{2\pi}\log|z|
 -\frac{\eps^{N+2}}{2\pi N}
   \Rea\!\left(M_N(q)z^{-N}\right)\notag\\
 &+\eps^{2N+2}\mathscr R_{\rm mp}(\eps,z,q).
 \label{eq:078}
\end{align}
For every integer $s\ge0$, the remainder and its shape derivative extend so
that
\[
 \mathscr R_{\rm mp}\in
 C\bigl([0,\eps_{\rm mp}]\times B_{X_P}(0,2r_P);
 C^s(\mathcal A)\bigr)
\]
and
\[
 D_q\mathscr R_{\rm mp}\in
 C\bigl([0,\eps_{\rm mp}]\times B_{X_P}(0,2r_P);
 \mathcal L(X_P,C^s(\mathcal A))\bigr).
\]
Moreover, there is $C_s>0$ such that
\[
 \sup_{\substack{0\le\eps\le\eps_{\rm mp}\\
 \|q\|_{X_P}\le r_P}}
 \left(
  \|\mathscr R_{\rm mp}(\eps,\cdot,q)\|_{C^s(\mathcal A)}
  +\|D_q\mathscr R_{\rm mp}(\eps,\cdot,q)\|_{
     \mathcal L(X_P,C^s(\mathcal A))}
 \right)\le C_s.
\]
\end{lemma}
\endgroup

\begingroup
\begin{proof}
For $|\eps y|<|z|$,
\begin{equation*}
 \log|z-\eps y|
 =\log|z|-\sum_{n\ge1}\frac{\eps^n}{n}
 \Rea(y^nz^{-n}).
\end{equation*}
After integration and use of \eqref{eq:027}, only the
indices $n=jN$ remain.  Extracting $j=1$ gives \eqref{eq:078}; more
explicitly,
\[
 \mathscr R_{\rm mp}(\eps,z,q)
 =-\frac1{2\pi}\sum_{j\ge2}\frac{\eps^{(j-2)N}}{jN}
   \Rea\!\left(M_{jN}(q)z^{-jN}\right).
\]
Choose $r_P$, $R_{\rm seed}\ge R_P$, and $\eps_{\rm mp}$ as in the statement, where \(R_P\) is the constant in
Lemma~\ref{lem:02}. The
boundary representation
\eqref{eq:031} and the already established estimate
\eqref{eq:028} control the moments and their shape
derivatives uniformly for $q\in B_{X_P}(0,2r_P)$.
Moreover, for $|\beta|\le s$,
\[
 |D_z^\beta z^{-jN}|
 \le C_s(1+jN)^s r_-^{-jN-s},
 \qquad z\in\mathcal A.
\]
The resulting geometric majorant is summable uniformly for
$0\le\eps\le\eps_{\rm mp}$ and $\|q\|_{X_P}\le r_P$, also after one
Fr\'echet derivative in $q$.  The same argument on every smaller closed ball
in $B_{X_P}(0,2r_P)$ gives the asserted continuity there.  Termwise
differentiation and the Weierstrass test prove the stated mapping properties. This proves the lemma.
\end{proof}
\endgroup

\subsection{\texorpdfstring{The outer equation}{The outer equation}}
\label{sec:outer}

For $\eps>0$ define the renormalized outer residual
\begin{equation}\label{eq:079}
 \mathscr F_{\rm o}(\eps,g,h,\lambda)
 :=t_{\rm o}^{-1}\cF_0
 \bigl(\eps,\Om_{\eps,\lambda},t_{\rm o}g,t_{\rm i}h\bigr), 
\end{equation}
where $\cF_0$ is defined in \eqref{eq:F0-def}. 
For $q$ near zero, set
\begin{equation}\label{eq:PN-def}
 P_N(q)(\theta)
 :=\frac1{2\pi N}\pa_\theta
 \Rea\!\left(M_N(q)e^{-\mathrm{i}N\theta}\right)
 =-\frac{M_N(q)}{2\pi}\sin(N\theta).
\end{equation}
Here $M_N(q)$ is the $N$-th area moment defined in
\eqref{eq:026}; it depends on $q$, but not on $\theta$.  Since
$P(q)=\overline{P(q)}$,
\[
 \overline{M_N(q)}
 =\int_{P(q)}\bar z^N\,\mathrm{d}A(z)
 =\int_{\overline{P(q)}}z^N\,\mathrm{d}A(z)=M_N(q),
\]
so $M_N(q)\in\R$ and the second identity in \eqref{eq:PN-def} follows.
Also define
\begin{equation}\label{eq:080}
 L_\mu\left(\sum_{k\ge1}g_{kN}\cos(kN\theta)\right)
 :=\sum_{k\ge1}\left(\frac12-kN\mu\right)
 g_{kN}\sin(kN\theta).
\end{equation}
Thus $P_N(q)\in Y_N$, and $q\mapsto P_N(q)$ is of class $C^1$.

\begingroup
\begin{proposition}[Outer expansion]\label{prop:03}
For $R>0$, set
\begin{equation}\label{eq:081}
 \mathbb B_R
 :=B_{X_N}(0,R)\times B_{X_P}(0,R)\times(-R,R).
\end{equation}
There exist $\eps_R\in(0,\eps_1]$ and $C_R>0$ such that the following
properties hold.
\begin{enumerate}[label=\textup{(\roman*)}]
\item For $0<\eps\le\eps_R$, the map $\mathscr F_{\rm o}$ defined by
\eqref{eq:079} is well defined on $\mathbb B_R$ and admits
a continuous extension
\[
 \mathscr F_{\rm o}:
 [0,\eps_R]\times\mathbb B_R\longrightarrow Y_N.
\]
For every fixed $\eps\in[0,\eps_R]$, this map is Fr\'echet $C^1$ in
$(g,h,\lambda)$, and
\[
 D_{(g,h,\lambda)}\mathscr F_{\rm o}
 \in C\!\left(
 [0,\eps_R]\times\mathbb B_R;
 \mathcal L(X_N\times X_P\times\R,Y_N)
 \right),
\]
where this derivative is continuous in operator norm.

Moreover, on $[0,\eps_R]\times\mathbb B_R$ one has
\begin{equation}\label{eq:082}
 \mathscr F_{\rm o}(\eps,g,h,\lambda)
 =L_{\delta-t_{\rm i}\lambda}g+P_N(t_{\rm i}h)
  +\eps^2\mathscr R_{\rm o}(\eps,g,h,\lambda),
\end{equation}
where
\[
 \mathscr R_{\rm o}
 \in C\bigl([0,\eps_R]\times\mathbb B_R;Y_N\bigr)
\]
and
\[
 D_{(g,h,\lambda)}\mathscr R_{\rm o}
 \in C\!\left(
 [0,\eps_R]\times\mathbb B_R;
 \mathcal L(X_N\times X_P\times\R,Y_N)
 \right).
\]
Moreover,
\begin{align}
 \sup_{\substack{0\le\eps\le\eps_R\\
 (g,h,\lambda)\in\mathbb B_R}}
 \Bigl(&\|\mathscr R_{\rm o}(\eps,g,h,\lambda)\|_{Y_N}
 \notag\\[-1mm]
 &+\|D_{(g,h,\lambda)}\mathscr R_{\rm o}
 (\eps,g,h,\lambda)\|_{
 \mathcal L(X_N\times X_P\times\R,Y_N)}
 \Bigr)
 \le C_R.
 \label{eq:083}
\end{align}
In particular,
\begin{equation}\label{eq:084}
 \mathscr F_{\rm o}(0,g,h,\lambda)=L_\delta g+P_N(0).
\end{equation}

\item The operator $L_\delta:X_N\to Y_N$ is an isomorphism and satisfies
\begin{equation}\label{eq:085}
 \|g\|_{X_N}
 \le C_{N,\delta,\alpha}\|L_\delta g\|_{Y_N},
 \qquad g\in X_N.
\end{equation}

\item The limiting outer equation
\[
 L_\delta g+P_N(0)=0
\]
has the unique solution
\begin{equation}\label{eq:g-star}
 g_*(\theta)=a_N\cos(N\theta),
 \qquad
 a_N=\frac{M_N(P)}{\pi(1-2N\delta)}>0.
\end{equation}
\end{enumerate}
\end{proposition}
\endgroup

\begin{proof}
\begingroup
\noindent\emph{Step 1. The Rankine contribution.}
Fix $R>0$ and put
\[
 I:=\left[\frac\delta2,\frac{3\delta}{2}\right].
\]
After shrinking an open neighborhood $\mathcal U_{\mathcal S}\subset X_N$
of zero, define the local map
\begin{equation}\label{eq:086}
 \mathcal S:\mathcal U_{\mathcal S}\times I\longrightarrow Y_N,
 \qquad
 \mathcal S(r,\mu)
 :=\pa_\theta\left[
 \psi_{D_0(r)}(z_r)-\frac{\frac12-\mu}{2}|z_r|^2\right].
\end{equation}
The map is $C^\infty$.  To verify this directly, let $D$ be a $C^1$ domain.
Choosing $\Phi(y)=\frac{\bar{y}-\bar{x}}{y-x}$ in \eqref{eq:025}, the Cauchy--Green formula  on $D\setminus\overline{B_\eta(x)}$, followed by
$\eta\downarrow0$, gives for $x\in\pa D$
\[
 \overline{\mathfrak g_D(x)}
 =\frac{\mathrm{i}}{4\pi}\int_{\pa D}
   \frac{\bar y-\bar x}{y-x}\,\mathrm{d}y,
 \qquad
 \mathfrak g_D=\pa_{x_1}\psi_D+\mathrm{i}\pa_{x_2}\psi_D.
\]
The additional circular arc and the omitted boundary piece are both $O(\eta)$.
Consequently,
\begin{equation}\label{eq:087}
 \pa_\theta\psi_{D_0(r)}(z_r(\theta))
 =\frac1{4\pi}\Rea\left\{
 \mathrm{i}z_r'(\theta)\int_0^{2\pi}
 \frac{\overline{z_r(\varphi)}-\overline{z_r(\theta)}}
      {z_r(\varphi)-z_r(\theta)}z_r'(\varphi)\,\mathrm{d}\varphi
 \right\}.
\end{equation}
For $z\in C^{2-\alpha}(\T)$, let
\[
 (\mathfrak Dz)(\theta,\varphi)
 :=\begin{cases}
 \displaystyle\frac{z(\varphi)-z(\theta)}
 {e^{\mathrm{i}\varphi}-e^{\mathrm{i}\theta}},&\varphi\ne\theta,\\[3mm]
 \displaystyle\frac{z'(\theta)}{\mathrm{i}e^{\mathrm{i}\theta}},
 &\varphi=\theta.
 \end{cases}
\]
Writing $\varphi=\theta+\tau$ expresses this quotient as the ratio of the
integrals of $z'$ and $\mathrm{i}e^{\mathrm{i}\cdot}$ along the segment;
the denominator has modulus at least $2/\pi$ for $|\tau|\le\pi$.  Thus
\begin{equation}\label{eq:088}
 \|\mathfrak Dz\|_{C^{1-\alpha}(\T^2)}
 \le C\|z\|_{C^{2-\alpha}(\T)}.
\end{equation}
If $Q_r=\mathfrak Dz_r$, then $Q_0=1$ and, locally in $r$,
\begin{equation}\label{eq:089}
 \frac{\overline{z_r(\varphi)}-\overline{z_r(\theta)}}
      {z_r(\varphi)-z_r(\theta)}
 =-e^{-\mathrm{i}(\theta+\varphi)}
   \frac{\overline{Q_r(\theta,\varphi)}}{Q_r(\theta,\varphi)}.
\end{equation}
The Banach-algebra property and smooth inversion in
$C^{1-\alpha}(\T^2)$ remove the apparent diagonal singularity in
\eqref{eq:087}, proving the stated local
$C^\infty$ mapping property.

At the disk, the thin-annulus formula and dominated convergence give
\[
 D_r\psi_{D_0(r)}(e^{\mathrm{i}\theta})\big|_{r=0}g
 =\frac1{2\pi}\int_0^{2\pi}
 \log|e^{\mathrm{i}\theta}-e^{\mathrm{i}\varphi}|
 g(\varphi)\,\mathrm{d}\varphi.
\]
The displacement of the evaluation point contributes $g/2$, and the
centrifugal term contributes $-(1/2-\mu)g$.  The sign of the remaining
Fourier multiplier follows from
\[
 \log|e^{\mathrm{i}\theta}-\varrho e^{\mathrm{i}\varphi}|
 =-\Rea\sum_{n\ge1}
   \frac{\varrho^n e^{\mathrm{i}n(\varphi-\theta)}}{n},
 \qquad 0<\varrho<1,
\]
which yields, after termwise integration and then $\varrho\uparrow1$,
\[
 \frac1{2\pi}\int_0^{2\pi}
 \log|e^{\mathrm{i}\theta}-e^{\mathrm{i}\varphi}|
 \cos(n\varphi)\,\mathrm{d}\varphi
 =-\frac1{2n}\cos(n\theta).
\]
Therefore
\begin{equation}\label{eq:090}
 D_r\mathcal S(0,\mu)g=L_\mu g.
\end{equation}

Since $\mathcal S(0,\mu)=0$, Taylor's formula gives
\begin{equation}\label{eq:091}
 t^{-1}\mathcal S(tg,\mu)
 =L_\mu g+t\mathscr R_{\mathcal S}(t,g,\mu),
\end{equation}
where
\begin{equation}\label{eq:092}
 \mathscr R_{\mathcal S}(t,g,\mu)
 :=\int_0^1(1-s)
 D_r^2\mathcal S(stg,\mu)[g,g]\,\mathrm{d}s.
\end{equation}
A finite cover of the compact interval $I$ supplies $\eta,C>0$ such that
\begin{equation}\label{eq:093}
 \|D_r^2\mathcal S(r,\mu)\|
 +\|D_r^3\mathcal S(r,\mu)\|
 +\|D_\mu D_r^2\mathcal S(r,\mu)\|
 \le C
 \quad\text{for }\|r\|_{X_N}\le\eta,\ \mu\in I.
\end{equation}
Choose $t_R>0$ with $t_RR\le\eta$.  Differentiating
\eqref{eq:092} gives
\begin{align}
 D_g\mathscr R_{\mathcal S}(t,g,\mu)[\dot g]
 &=\int_0^1(1-s)\Bigl(
 D_r^2\mathcal S(stg,\mu)[\dot g,g]
 +D_r^2\mathcal S(stg,\mu)[g,\dot g]\notag\\
 &\hspace{7em}+stD_r^3\mathcal S(stg,\mu)[\dot g,g,g]
 \Bigr)\,\mathrm{d}s,
 \label{eq:094}\\
 D_\mu\mathscr R_{\mathcal S}(t,g,\mu)[\dot\mu]
 &=\int_0^1(1-s)
 \bigl(D_\mu D_r^2\mathcal S(stg,\mu)[g,g]\bigr)
 \dot\mu\,\mathrm{d}s.
 \label{eq:095}
\end{align}
Hence
\begin{equation}\label{eq:096}
 \sup_{\substack{|t|\le t_R,\ \|g\|_{X_N}\le R\\ \mu\in I}}
 \bigl(\|\mathscr R_{\mathcal S}\|+
 \|D_g\mathscr R_{\mathcal S}\|+
 \|D_\mu\mathscr R_{\mathcal S}\|\bigr)\le C_R,
\end{equation}
with the norms taken in the spaces stated in the proposition.

\medskip
\noindent\emph{Step 2. The contribution of the holes.}
Set
\[
 t=t_{\rm o}=\eps^{N+2},\qquad
 q=t_{\rm i}h=\eps^{2N}h,\qquad
 \mu=\delta-t_{\rm i}\lambda,
 \qquad
 z_{t,g}=\sqrt{1+2tg}\,e^{\mathrm{i}\theta}.
\]
Introduce the single difference quotient
\begin{equation}\label{eq:097}
 \Theta_N(t,g):=
 \begin{cases}
 \displaystyle\frac{(1+2tg)^{-N/2}-1}{t},&t\ne0,\\[2mm]
 -Ng,&t=0.
 \end{cases}
\end{equation}
It is $C^1$ on every tube where $1+2tg$ stays positive because
\begin{align*}
 \Theta_N(t,g)
 &=-Ng\int_0^1(1+2stg)^{-N/2-1}\,\mathrm{d}s,\\
 D_g\Theta_N(t,g)[\dot g]
 &=-N(1+2tg)^{-N/2-1}\dot g.
\end{align*}
The far-field expansion \eqref{lem:04} and
$\pa_\theta\log|z_{t,g}|=tg'/(1+2tg)$ give the exact identity
\begin{equation}\label{eq:098}
 -\frac1{t_{\rm o}}\pa_\theta
 \psi_{\eps P(q)}(z_{t,g})
 =P_N(q)+\eps^2\mathscr R_H(\eps,g,h),
\end{equation}
where $P_N$ is defined by \eqref{eq:PN-def} and
\begin{align}
 \mathscr R_H(\eps,g,h)
 :={}&-\frac{\mathscr A(q)}{2\pi}\frac{g'}{1+2tg}
 +\frac{\eps^N}{2\pi N}\pa_\theta\Rea\!\left(
 M_N(q)\Theta_N(t,g)e^{-\mathrm{i}N\theta}\right)\notag\\
 &-\eps^{N-2}\pa_\theta
 \mathscr R_{\rm mp}(\eps,z_{t,g},q).
 \label{eq:099}
\end{align}
The three terms in \eqref{eq:099} represent,
respectively, the shape-dependent monopole contribution, the evaluation
correction to the $N$-th multipole, and the higher-multipole tail.
Combining \eqref{eq:091} and
\eqref{eq:098} proves \eqref{eq:082} with
\begin{equation}\label{eq:100}
 \mathscr R_{\rm o}(\eps,g,h,\lambda)
 =\eps^N\mathscr R_{\mathcal S}
 (t_{\rm o},g,\delta-t_{\rm i}\lambda)
 +\mathscr R_H(\eps,g,h).
\end{equation}

We now choose the parameter range on the fixed product ball.  Apply
Lemma~\ref{lem:04} on $\{1/2\le|z|\le3/2\}$ and denote its shape
radius by $r_P$.  Decrease $\eps_R>0$ so that, uniformly on
$[0,\eps_R]\times\mathbb B_R$,
\[
 t_{\rm o}g\in\mathcal U_{\mathcal S}\cap\mathcal U_N,
 \quad t_{\rm i}h\in\mathcal U_P,
 \quad\|t_{\rm i}h\|_{X_P}\le r_P,
 \quad\mu\in I,
 \quad t_{\rm o}\le t_R,
\]
$1+2t_{\rm o}g$ is uniformly positive,
$z_{t_{\rm o},g}(\T)\subset\{1/2\le|z|\le3/2\}$, and
$\eps\overline{P(t_{\rm i}h)}\Subset D_0(t_{\rm o}g)$.  Thus the original
residual is well defined for $\eps>0$ on this fixed domain.

The chain factors introduced by the weighted variables are
\begin{equation}\label{eq:101}
 D_gz_{t_{\rm o},g}[\dot g]
 =\frac{t_{\rm o}\dot g}{\sqrt{1+2t_{\rm o}g}}
 e^{\mathrm{i}\theta},
 \qquad D_hq=t_{\rm i}I,
 \qquad D_\lambda\mu=-t_{\rm i}.
\end{equation}
Lemma~\ref{lem:02}, Lemma~\ref{lem:04} with $s=3$,
the H\"older Banach-algebra estimate, and
\eqref{eq:096} therefore give the uniform
bound \eqref{eq:083}.  They also give operator-norm
continuity of the derivative.  The first two terms in
\eqref{eq:099} are controlled by the $C^1$ maps
$\mathscr A$ and $M_N$.  The last term is controlled by the spatial
derivatives supplied by the $C^3$-bound in Lemma~\ref{lem:04} and by
the first Fr\'echet derivative of $\mathscr R_{\rm mp}$ with respect to the
shape variable.

 Define
\begin{equation}\label{eq:102}
 \mathscr R_{\rm o}(0,g,h,\lambda)
 :=-\frac{|P|}{2\pi}g'
 -\one_{\{N=2\}}\pa_\theta
  \mathscr R_{\rm mp}(0,e^{\mathrm{i}\theta},0).
\end{equation}
Then
\begin{equation}\label{eq:103}
 \begin{aligned}
 D_g\mathscr R_{\rm o}(0,g,h,\lambda)[\dot g]
 &=-\frac{|P|}{2\pi}\dot g',\\
 D_h\mathscr R_{\rm o}(0,g,h,\lambda)&=0,
 \qquad
 D_\lambda\mathscr R_{\rm o}(0,g,h,\lambda)=0.
 \end{aligned}
\end{equation}
The tail in \eqref{eq:099} has the order-one prefactor
$\eps^{N-2}=1$ when $N=2$, which produces the second term of
\eqref{eq:102}; for $N\ge3$ that prefactor tends to
zero.  Thus, when $N=2$, it retains this tail, whereas its
$g$- and $h$-derivatives vanish because \eqref{eq:101}
supplies $t_{\rm o}$ and $t_{\rm i}$, respectively.  The first term in
\eqref{eq:100} has the prefactor $\eps^N$, and
its $\lambda$ derivative has the additional factor $-t_{\rm i}$.  The
remaining $h$ derivatives carry $t_{\rm i}$, while the $N$-moment correction
also carries $\eps^N$.  These facts prove uniform convergence of the
remainder and its unknown derivative in operator norm, including at
$\eps=0$, and establish the first item of the proposition.

\medskip
\noindent\emph{Step 3. Fourier inversion and the limiting root.}
For
\[
 y(\theta)=\sum_{k\ge1}y_{kN}\sin(kN\theta)\in Y_N,
\]
nonresonance gives the explicit inverse
\begin{equation}\label{eq:104}
 (L_\delta^{-1}y)(\theta)
 =\sum_{k\ge1}\frac{y_{kN}}{\frac12-kN\delta}
 \cos(kN\theta).
\end{equation}
For $m=kN$,
\[
 \frac1{\frac12-\delta m}
 =-\frac1{\delta m}
 +\frac1{2\delta m(\frac12-\delta m)}.
\]
The two terms on the right are periodic multipliers of orders $-1$ and
$-2$ on the nonresonant $N$-fold subspace.  The multiplier estimate
therefore proves \eqref{eq:085} and the bounded bijectivity of
$L_\delta:X_N\to Y_N$.  Finally,
$P_N(0)=-M_N(P)\sin(N\theta)/(2\pi)$, so the unique limiting root is
\[
 g_*(\theta)
 =\frac{M_N(P)}{\pi(1-2N\delta)}\cos(N\theta).
\]
Its coefficient is positive by Proposition~\ref{prop:02} and
$\delta<1/(2N)$, proving \eqref{eq:g-star}.
\endgroup
\end{proof}

\subsection{\texorpdfstring{Harmonic feedback at the inner scale}
{Harmonic feedback at the inner scale}}
\label{sec:inner}

The outer deformation is small on the unit scale, but it is sampled at points
of size $\eps$ in the inner equation.  For $t$ near zero and $g\in X_N$, put
\begin{equation}\label{eq:Wtg-def}
 W_{t,g}(x):=
 \begin{cases}
 \displaystyle\frac{\psi_{D_0(tg)}(x)-\psi_\D(x)}{t},&t\ne0,\\[2mm]
 \displaystyle\frac1{2\pi}\int_0^{2\pi}
 g(\varphi)\log|x-e^{\mathrm{i}\varphi}|\,\mathrm{d}\varphi,&t=0.
 \end{cases}
\end{equation}
For $|x|\le1/2$ this function is harmonic.  Its dihedral invariance means
specifically that
\begin{equation}\label{eq:105}
 W_{t,g}(\zeta z)=W_{t,g}(z),
 \qquad
 W_{t,g}(\bar z)=W_{t,g}(z).
\end{equation}
\begingroup
\begin{lemma}\label{lem:05}
For $R>0$, set
\[
 \mathbb B_R^{\rm jet}:=B_{X_N}(0,R)\times B_{X_P}(0,R).
\]
There exist $\eps_R>0$ and $C_R>0$ such that, for
$0<\eps\le\eps_R$ and $(g,h)\in\mathbb B_R^{\rm jet}$,
\begin{align}
 &\eps^{-N}\left[
 W_{t_{\rm o},g}\bigl(\eps p_{t_{\rm i}h}(w)\bigr)
 -W_{t_{\rm o},g}(0)\right]\notag\\
 &\hspace{25mm}
 =-\frac{g_N}{2N}\Rea\bigl(p_0(w)^N\bigr)
  +\eps^N\mathscr R_{\rm jet}(\eps,g,h)(w),
 \label{eq:106}
\end{align}
where $g_N$ is the coefficient of $\cos(N\theta)$ in $g$, and
$\mathscr R_{\rm jet}(\eps,\cdot)$ is Fr\'echet $C^1$ on
$\mathbb B_R^{\rm jet}$.  Moreover,
\begin{equation}\label{eq:107}
 \sup_{\substack{0<\eps\le\eps_R\\(g,h)\in\mathbb B_R^{\rm jet}}}
 \left(
 \|\mathscr R_{\rm jet}(\eps,g,h)\|_{C^{2-\alpha}}
 +\|D_{(g,h)}\mathscr R_{\rm jet}(\eps,g,h)\|_{
   \mathcal L(X_N\times X_P,C^{2-\alpha})}
 \right)\le C_R.
\end{equation}
If the entire left-hand side of \eqref{eq:106} is denoted
by $\mathscr J_\eps(g,h)$ and one sets
\[
 \mathscr J_0(g,h):=-\frac{g_N}{2N}\Rea(p_0^N),
\]
then
\begin{align*}
 \mathscr J&\in C([0,\eps_R]\times\mathbb B_R^{\rm jet};C^{2-\alpha}),\\
 D_{(g,h)}\mathscr J&\in C\!\left(
 [0,\eps_R]\times\mathbb B_R^{\rm jet};
 \mathcal L(X_N\times X_P,C^{2-\alpha})\right).
\end{align*}
\end{lemma}
\endgroup

\begin{proof}
Fix $R>0$.  All constants below are chosen uniformly for
$(g,h)\in\mathbb B_R^{\rm jet}$, after decreasing $\eps_R$ when necessary.
Polar integration in the thin region between the unit circle and the graph
$z_{tg}$ gives the separated-kernel formula
\begin{align}
 W_{t,g}(x)
 =\frac1{2\pi}\int_0^{2\pi}g(\varphi)\int_0^1
 \log\left|x-\sqrt{1+2\sigma t g(\varphi)}e^{\mathrm{i}\varphi}\right|
 \,\mathrm{d}\sigma\,\mathrm{d}\varphi.
 \label{eq:108}
\end{align}
For $|x|\le1/2$ the two arguments of the logarithm remain uniformly
separated.  Hence $(t,g)\mapsto W_{t,g}$ is $C^\infty$ with values in
$C^k(\overline{B_{1/2}})$ for every fixed $k$.  Its derivatives are
uniformly bounded when $|t|$ is small and $g\in B_{X_N}(0,R)$.

At $t=0$, expansion of the logarithm in complex powers gives
\begin{equation}\label{eq:109}
 W_{0,g}(z)-W_{0,g}(0)
 =-\sum_{k\ge1}\frac{g_{kN}}{2kN}\Rea(z^{kN}),
 \qquad |z|<1.
\end{equation}
For nonzero $t$, harmonicity and \eqref{eq:105} give the
convergent expansion
\begin{equation}\label{eq:110}
 W_{t,g}(z)-W_{t,g}(0)
 =\sum_{k\ge1}a_{kN}(t,g)\Rea(z^{kN}),
 \qquad |z|<\frac12.
\end{equation}
Indeed, the harmonic function on the left is the real part of a
holomorphic function on $B_{1/2}$.  Its power series there, together with
rotation by $\zeta$ and reflection, contains only the degrees $kN$ and has
real coefficients.
Fix $r_0\in(0,1/2)$.  Orthogonality on $|z|=r_0$ gives the coefficient formula
\begin{equation}\label{eq:111}
 a_{kN}(t,g)
 =\frac1{\pi r_0^{kN}}\int_0^{2\pi}
 \bigl[W_{t,g}(r_0e^{\mathrm{i}\theta})-W_{t,g}(0)\bigr]
 \cos(kN\theta)\,\mathrm{d}\theta.
\end{equation}
Formula \eqref{eq:109}, \eqref{eq:111}, and
the $C^\infty$ parameter dependence imply
\begin{equation}\label{eq:112}
 a_N(t,g)=-\frac{g_N}{2N}+O(t),
 \qquad
 |a_{kN}(t,g)|\le Cr_0^{-kN}\qquad(k\ge2).
\end{equation}
More precisely, let
\[
 \ell_N:X_N\to\R,
 \qquad \ell_N(g)=g_N.
\]
There is a constant $C_R$ such that, uniformly for $|t|$ small and
$g\in B_{X_N}(0,R)$,
\begingroup
\begin{align}
 &\left|a_N(t,g)+\frac{g_N}{2N}\right|
 +\left\|D_ga_N(t,g)+\frac1{2N}\ell_N\right\|_{
   \mathcal L(X_N,\R)}
 \le C_R|t|,
 \label{eq:113}\\
 &|a_{kN}(t,g)|
 +\|D_ga_{kN}(t,g)\|_{\mathcal L(X_N,\R)}
 \le C_R r_0^{-kN},\qquad k\ge2.
 \label{eq:114}
\end{align}
\endgroup
The first estimate follows from the identities at $t=0$ and the mean-value
formula in $t$, applied both to $a_N$ and to $D_ga_N$.  The second follows by
applying \eqref{eq:111} to $W_{t,g}$ and to
$D_gW_{t,g}[\dot g]$, using the uniform fixed-circle bounds supplied by
\eqref{eq:108}.

Insert $t=t_{\rm o}$ and
$p_{t_{\rm i}h}=p_0+\rho_*^2t_{\rm i}h$ in
\eqref{eq:110}.  After division by $\eps^N$, the $k=1$ term
is the first term on the right-hand side of
\eqref{eq:106}; its coefficient error is
$O(t_{\rm o})$, and its evaluation error is $O(t_{\rm i})$.  The sum over
$k\ge2$ is $O(\eps^N)$.  Since
$t_{\rm o}=O(\eps^N)$ and $t_{\rm i}=O(\eps^N)$ for $N\ge2$, the stated
estimate follows.  Choose $B_R\ge1$ so that
$\|p_{t_{\rm i}h}\|_{C^{2-\alpha}}\le B_R$ for
$h\in B_{X_P}(0,R)$ and $0\le\eps\le\eps_R$.  The H\"older algebra estimate
and its differentiated version
give
\begingroup
\[
 \|p^m\|_{C^{2-\alpha}}
 +\|D_p(p^m)\|_{\mathcal L(C^{2-\alpha},C^{2-\alpha})}
 \le C_R m^2B_R^m.
\]
\endgroup
After reducing $\eps_R$ so that $\eps_R B_R<r_0/2$, the geometric series obtained
from \eqref{eq:114} is uniformly summable, also after a
$g$ or $h$ derivative.  Thus, if
\begin{align*}
 E_\eps(g,h):={}&\eps^{-N}\left[
 W_{t_{\rm o},g}\bigl(\eps p_{t_{\rm i}h}\bigr)
 -W_{t_{\rm o},g}(0)\right]
 +\frac{g_N}{2N}\Rea(p_0^N),
\end{align*}
then
\begingroup
\begin{equation}\label{eq:115}
 \|E_\eps(g,h)\|_{C^{2-\alpha}}
 +\|D_{(g,h)}E_\eps(g,h)\|_{
   \mathcal L(X_N\times X_P,C^{2-\alpha})}
 \le C_R\eps^N.
\end{equation}
\endgroup
Here $D_hp_{t_{\rm i}h}[\dot h]=\rho_*^2t_{\rm i}\dot h$, so the
$h$ derivative has the same order.  Estimate
\eqref{eq:115}, together with the $C^\infty$
separated-kernel formula, proves
\eqref{eq:107} and the two asserted continuity
statements for $\mathscr J$.
\end{proof}

\subsection{\texorpdfstring{The inner equation}{The inner equation}}

\begingroup
Recall the augmented template operator from
Proposition~\ref{prop:02}\textnormal{(iii)}:
\begin{equation}\label{eq:A-def}
 \mathscr A_P:X_P\times\R\longrightarrow Y_P,
 \qquad
 \mathscr A_P[h,\lambda]
 =-D_q\cG(\delta,0)h+D_\mu\cG(\delta,0)\lambda.
\end{equation}
It is a bounded linear isomorphism by
\eqref{eq:063}.
\endgroup
Define also
\begin{equation}\label{eq:CN-def}
 C_Ng
 :=\pa_\theta\left[-\frac{g_N}{2N}
 \Rea\bigl(p_0(e^{\mathrm{i}\theta})^N\bigr)\right].
\end{equation}
The coefficient map $g\mapsto g_N$ is bounded on $X_N$; hence
$C_N:X_N\to Y_P$ is bounded.

For $\eps>0$ let
\begin{equation}\label{eq:116}
 \mathscr F_{\rm i}(\eps,g,h,\lambda)
 :=t_{\rm i}^{-1}\cF_1
 \bigl(\eps,\Om_{\eps,\lambda},t_{\rm o}g,t_{\rm i}h\bigr), 
\end{equation}
where $\cF_1$ is defined by \eqref{eq:F1-def}. 

\begingroup
\begin{proposition}[Inner expansion]\label{prop:04}
Put $Z:=X_N\times X_P\times\R$.
For every $R>0$ there exists $\eps_R>0$ such that the following properties
hold on the product ball $\mathbb B_R$ defined in
\eqref{eq:081}.
For $0<\eps\le\eps_R$, the residual
\eqref{eq:116} is well defined on $\mathbb B_R$ and admits
a continuous extension
\[
 \mathscr F_{\rm i}:[0,\eps_R]\times\mathbb B_R\longrightarrow Y_P.
\]
For each fixed $\eps\in[0,\eps_R]$ it is Fr\'echet $C^1$ in
$(g,h,\lambda)$, and
\[
 D_{(g,h,\lambda)}\mathscr F_{\rm i}
 \in C\!\left([0,\eps_R]\times\mathbb B_R;
 \mathcal L(Z,Y_P)\right),
\]
where this derivative is continuous in operator norm.

Moreover, one has
\begin{equation}\label{eq:117}
 \mathscr F_{\rm i}(\eps,g,h,\lambda)
 =C_Ng+\mathscr A_P[h,\lambda]+\mathscr R_{\rm i}(\eps,g,h,\lambda),
\end{equation}
where $\mathscr R_{\rm i}(0,\cdot)=0$ and
\begin{equation}\label{eq:118}
 \lim_{\eps\downarrow0}
 \sup_{(g,h,\lambda)\in\mathbb B_R}
 \left(
 \|\mathscr R_{\rm i}(\eps,g,h,\lambda)\|_{Y_P}
 +\|D_{(g,h,\lambda)}\mathscr R_{\rm i}(\eps,g,h,\lambda)\|_{
   \mathcal L(Z,Y_P)}
 \right)=0.
\end{equation}

In particular, the limiting map is
\begin{equation}\label{eq:119}
 \mathscr F_{\rm i}(0,g,h,\lambda)=C_Ng+\mathscr A_P[h,\lambda].
\end{equation}
\end{proposition}
\endgroup

\begin{proof}
Use the decomposition
\[
 \psi_{D_0(t_{\rm o}g)}
 =\psi_\D+t_{\rm o}W_{t_{\rm o},g}.
\]
Put $p=p_{t_{\rm i}h}(e^{\mathrm{i}\theta})$.  Fix $R>0$.  After
decreasing $\eps_R$, one has $t_{\rm i}h\in\mathcal V_P$,
$|t_{\rm i}\lambda|<\eta_\mu$, and
$\eps p\in\D$ uniformly for $(g,h,\lambda)\in\mathbb B_R$ and
$0<\eps\le\eps_R$.
Therefore the Rankine formula \eqref{eq:072} gives
\begin{equation}\label{eq:120}
 \frac{\psi_\D(\eps p)-\psi_\D(0)}{\eps^2}
 =\frac14|p|^2.
\end{equation}
Since $\psi_{D_0(t_{\rm o}g)}=\psi_\D+t_{\rm o}W_{t_{\rm o},g}$,
\begin{equation}\label{eq:121}
 Q_{\eps,t_{\rm o}g}(p)
 =\frac14|p|^2
 +\frac{t_{\rm o}}{\eps^2}
  \bigl[W_{t_{\rm o},g}(\eps p)-W_{t_{\rm o},g}(0)\bigr].
\end{equation}
Furthermore,
\begin{equation}\label{eq:122}
 \frac14-\frac{\Om_{\eps,\lambda}}2
 =\frac{\delta-t_{\rm i}\lambda}{2},
 \qquad
 \frac{t_{\rm o}}{\eps^2t_{\rm i}}=\eps^{-N}.
\end{equation}
By Proposition~\ref{prop:02}\textnormal{(iii)}, with
$\mu=\delta-t_{\rm i}\lambda$ and $q=t_{\rm i}h$,
\[
 \pa_\theta\left[
  \frac{\delta-t_{\rm i}\lambda}{2}|p|^2
  -\psi_{P(t_{\rm i}h)}(p)\right]
 =-\cG(\delta-t_{\rm i}\lambda,t_{\rm i}h).
\]
Combining these identities gives the exact decomposition
\begin{align}
 \mathscr F_{\rm i}(\eps,g,h,\lambda)
 ={}&-\frac1{t_{\rm i}}
 \cG(\delta-t_{\rm i}\lambda,t_{\rm i}h)\notag\\
 &+\eps^{-N}\pa_\theta\left[
 W_{t_{\rm o},g}\bigl(\eps p_{t_{\rm i}h}\bigr)
 -W_{t_{\rm o},g}(0)\right].
 \label{eq:123}
\end{align}
The constant $\eps^2|P(t_{\rm i}h)|\log\eps/(2\pi)$ produced by scaling the
inner potential has disappeared under $\pa_\theta$; therefore no
$|\log\eps|$ loss occurs, even after differentiation in $h$.

Since $\cG(\delta,0)=0$, the first term on the right-hand side of
\eqref{eq:123} equals
\begin{align}
 -\int_0^1\bigl\{
 D_q\cG(\delta-st_{\rm i}\lambda,st_{\rm i}h)h
 -D_\mu\cG(\delta-st_{\rm i}\lambda,st_{\rm i}h)\lambda
 \bigr\}\,\mathrm{d}s.
 \label{eq:124}
\end{align}
It converges to $\mathscr A_P[h,\lambda]$.  For $t_{\rm i}>0$, its
derivative in a variation $(\dot h,\dot\lambda)$ is
\[
 -D_q\cG(\delta-t_{\rm i}\lambda,t_{\rm i}h)\dot h
 +D_\mu\cG(\delta-t_{\rm i}\lambda,t_{\rm i}h)\dot\lambda,
\]
which converges in operator norm to
$\mathscr A_P[\dot h,\dot\lambda]$.  This argument uses the $C^1$ regularity
of the template functional and no higher shape derivative.  The second term in
\eqref{eq:123} is covered by
Lemma~\ref{lem:05}.  Applying $\pa_\theta$ gives $C_Ng$.
Because $(\delta-st_{\rm i}\lambda,st_{\rm i}h)$ converges uniformly
to $(\delta,0)$ on $[0,1]\times\mathbb B_R$, continuity of $D\cG$ at
$(\delta,0)$ makes the seed remainder and its unknown derivative tend to
zero uniformly on $\mathbb B_R$.  The uniform estimate
\eqref{eq:107} gives the same conclusion for the
harmonic term.  This proves \eqref{eq:118} and all asserted
mapping properties.
\end{proof}

\section{The coupled implicit-function argument}
\label{sec:gluing}

Set
\begin{equation}\label{eq:125}
 \mathscr F(\eps,g,h,\lambda)
 :=\bigl(\mathscr F_{\rm o}(\eps,g,h,\lambda),
          \mathscr F_{\rm i}(\eps,g,h,\lambda)\bigr)
\end{equation}
as a map into $Y_N\times Y_P$.  Propositions~\ref{prop:03}
and~\ref{prop:04} show that it is
continuous for $\eps\ge0$, is $C^1$ in the unknowns, and has a jointly
continuous derivative in those variables.  Its limiting equation is the
affine system
\begin{equation}\label{eq:126}
 \mathscr F(0,g,h,\lambda)
 =\begin{pmatrix}
   L_\delta g+P_N(0)\\
   C_Ng+\mathscr A_P[h,\lambda]
  \end{pmatrix}.
\end{equation}
Define
\begin{equation}\label{eq:127}
 g_*:=-L_\delta^{-1}P_N(0),
 \qquad
 (h_*,\lambda_*):=-\mathscr A_P^{-1}C_Ng_*.
\end{equation}
Thus $g_*$ is given explicitly by \eqref{eq:g-star} and
$\mathscr F(0,g_*,h_*,\lambda_*)=0$.

\begingroup
We first choose a product neighborhood centered at the limiting root for the
parameter-dependent implicit-function theorem.  For $R>0$, set
\begin{equation}\label{eq:128}
 \mathcal V_R
 :=B_{X_N}(g_*,R)\times B_{X_P}(h_*,R)
   \times(\lambda_*-R,\lambda_*+R).
\end{equation}
To place this neighborhood in a ball centered at the origin on which the uniform
estimates apply, choose
\[
 R_0>\max\{\|g_*\|_{X_N}+R,\|h_*\|_{X_P}+R,|\lambda_*|+R\}.
\]
Then $\overline{\mathcal V_R}\subset\mathbb B_{R_0}$.  Since
$\mathcal U_N$ and $\mathcal V_P$ are open
neighborhoods of the origin, we may decrease $\eps_1>0$ so that, uniformly
for $(g,h,\lambda)\in\overline{\mathcal V_R}$ and
$0\le\eps\le\eps_1$,
\begin{equation}\label{eq:129}
 t_{\rm o}g\in\mathcal U_N,\qquad
 t_{\rm i}h\in\mathcal V_P,
 \qquad |t_{\rm i}\lambda|<\eta_\mu,
\end{equation}
where $(\delta-\eta_\mu,\delta+\eta_\mu)$ is a fixed interval on which the
template functional is defined.  Shrinking $\eps_1$ again if necessary, the
uniformly bounded sets $P(t_{\rm i}h)$ satisfy
\begin{equation}\label{eq:130}
 \eps\overline{P(t_{\rm i}h)}\Subset D_0(t_{\rm o}g),
\end{equation}
and the physical boundary components remain embedded and pairwise separated,
uniformly for $0<\eps\le\eps_1$.  Thus the unrenormalized boundary residuals are defined on
$(0,\eps_1]\times\mathcal V_R$, and their renormalized continuous extensions
define $\mathscr F$ on the single fixed set
$[0,\eps_1]\times\mathcal V_R$.  Applying
Propositions~\ref{prop:03} and~\ref{prop:04} with
$R_0$, and
decreasing $\eps_1$ below their corresponding thresholds, gives there
\[
 \mathscr F\in C([0,\eps_1]\times\mathcal V_R;Y_N\times Y_P),
\]
and
\[
 D_{(g,h,\lambda)}\mathscr F
 \in C\bigl([0,\eps_1]\times\mathcal V_R;
 \mathcal L(X_N\times X_P\times\R,Y_N\times Y_P)\bigr).
\]
This derivative is continuous in operator norm.  This verifies the fixed-domain
hypotheses of Lemma~\ref{lem:01}; no $\eps$-dependent
neighborhood is used below.
\endgroup

The derivative at this root is
\begin{equation}\label{eq:131}
D_{(g,h,\lambda)}\mathscr F(0,g_*,h_*,\lambda_*)
 =\begin{pmatrix}
   L_\delta&0\\
   C_N&\mathscr A_P
  \end{pmatrix}.
\end{equation}
The upper-right block vanishes because the limiting outer equation contains
neither $h$ nor $\lambda$; Proposition~\ref{prop:03} gives this
conclusion at the derivative level through its joint operator-norm
continuity.
Both diagonal blocks are isomorphisms.  Hence the full operator is an
isomorphism; explicitly, its inverse sends $(y_{\rm o},y_{\rm i})$ to
\begin{equation}\label{eq:132}
 \left(
 L_\delta^{-1}y_{\rm o},
 \mathscr A_P^{-1}\bigl[y_{\rm i}-C_NL_\delta^{-1}y_{\rm o}\bigr]
 \right).
\end{equation}

\begingroup
\begin{proposition}\label{prop:05}
Let $x_*:=(g_*,h_*,\lambda_*)\in Z$.
There exist $\eps_0>0$, an open neighborhood
$V\subset\mathcal V_R$ of $x_*$, and a unique continuous map
\begin{equation}\label{eq:133}
 [0,\eps_0)\ni\eps\longmapsto
 x_\eps=(g_\eps,h_\eps,\lambda_\eps)\in V
\end{equation}
such that
\begin{equation}\label{eq:134}
 x_0=x_*,
 \qquad
 x_\eps\longrightarrow x_*
 \quad\text{in }Z\quad(\eps\downarrow0),
\end{equation}
and
\begin{equation}\label{eq:135}
 \mathscr F(\eps,g_\eps,h_\eps,\lambda_\eps)=0,
 \qquad 0\le\eps<\eps_0.
\end{equation}
For each fixed $\eps\in[0,\eps_0)$, $x_\eps$ is the only zero of
$\mathscr F(\eps,\cdot)$ in $V$.  Moreover, there is $C>0$ such that
\begin{equation}\label{eq:g-rate}
 \|g_\eps-g_*\|_{X_N}\le C\eps^2,
 \qquad 0\le\eps<\eps_0.
\end{equation}
\end{proposition}
\endgroup

\begin{proof}
Apply Lemma~\ref{lem:01}, the continuous-parameter
implicit-function theorem, to \eqref{eq:125} at the root
\eqref{eq:127}, using \eqref{eq:131}.  This gives
\eqref{eq:134}--\eqref{eq:135} and the stated
local uniqueness in $V$.

For the rate, take the first component of \eqref{eq:135} and use
\eqref{eq:082}.  We obtain
\[
 0=L_{\delta-t_{\rm i}\lambda_\eps}g_\eps
   +P_N(t_{\rm i}h_\eps)
   +\eps^2\mathscr R_{\rm o}
     (\eps,g_\eps,h_\eps,\lambda_\eps).
\]
The limiting outer equation is
\[
 0=L_\delta g_*+P_N(0).
\]
Subtracting these identities gives
\begin{align*}
 L_\delta(g_\eps-g_*)={}&
 -\bigl(L_{\delta-t_{\rm i}\lambda_\eps}-L_\delta\bigr)g_\eps\\
 &-\bigl(P_N(t_{\rm i}h_\eps)-P_N(0)\bigr)
 -\eps^2\mathscr R_{\rm o}
   (\eps,g_\eps,h_\eps,\lambda_\eps).
\end{align*}
By \eqref{eq:080},
\[
 \bigl(L_{\delta-t_{\rm i}\lambda}-L_\delta\bigr)g
 =-t_{\rm i}\lambda\,\pa_\theta g.
\]
Because \eqref{eq:134} implies that
$(g_\eps,h_\eps,\lambda_\eps)$ remains bounded,
\[
 \left\|
 \bigl(L_{\delta-t_{\rm i}\lambda_\eps}-L_\delta\bigr)g_\eps
 \right\|_{Y_N}\le Ct_{\rm i}.
\]
The $C^1$ regularity of $P_N:X_P\to Y_N$ also gives
\[
 P_N(t_{\rm i}h_\eps)-P_N(0)
 =t_{\rm i}\int_0^1
 DP_N(st_{\rm i}h_\eps)[h_\eps]\,\mathrm{d}s,
\]
and hence
\[
 \|P_N(t_{\rm i}h_\eps)-P_N(0)\|_{Y_N}\le Ct_{\rm i}.
\]
Finally, the branch lies in the ball $\mathbb B_{R_0}$ centered at the
origin, so the uniform estimate
\eqref{eq:083} yields
\[
 \left\|\eps^2\mathscr R_{\rm o}
 (\eps,g_\eps,h_\eps,\lambda_\eps)\right\|_{Y_N}
 \le C\eps^2.
\]
Consequently, since $t_{\rm i}=\eps^{2N}$ and $N\ge2$,
\[
 \|L_\delta(g_\eps-g_*)\|_{Y_N}
 \le C(t_{\rm i}+\eps^2)\le C\eps^2.
\]
Applying \eqref{eq:085}, we conclude that
\[
 \|g_\eps-g_*\|_{X_N}\le C\eps^2,
\]
which proves \eqref{eq:g-rate}.
\end{proof}

\section{\texorpdfstring{Geometry and quantitative asymptotics}
{Geometry and quantitative asymptotics}}
\label{sec:05}

\begin{proof}[Proof of Theorem~\ref{thm:main}]
Fix the nonresonant microscopic parameter $\rho_*$ chosen in
\eqref{eq:067}, and henceforth regard $P$, $p_0$, and $\delta$ as
fixed.  Let $(g_\eps,h_\eps,\lambda_\eps)$ be the branch in
Proposition~\ref{prop:05}.  Define
\begin{align}
 G_\eps&:=D_0(t_{\rm o}g_\eps),\notag\\
 H_{j,\eps}&:=\eps P_j(t_{\rm i}h_\eps),
 \qquad 1\le j\le N,\notag\\
 D_\eps&:=G_\eps\setminus
 \bigcup_{j=1}^N\overline{H_{j,\eps}},\notag\\
 \Om_\eps&:=\Om_*+t_{\rm i}\lambda_\eps.
 \label{eq:136}
\end{align}

The reference components $P_j$ are embedded and separated by a positive
distance.  Since $t_{\rm i}h_\eps\to0$ in $C^1$, embeddedness and separation
persist.  After multiplication by $\eps$ the holes lie in a ball of radius
$O(\eps)$, whereas $\pa G_\eps$ converges in $C^1$ to the unit circle.  Hence,
after decreasing $\eps_0$, the closed holes are pairwise disjoint and compactly
contained in $G_\eps$.  Since $G_\eps$ and the $H_{j,\eps}$ are Jordan domains,
with pairwise disjoint closures $\overline{H_{j,\eps}}\Subset G_\eps$, the
finite-component consequence of the Jordan--Schoenflies theorem
(equivalently, its successive application to the boundary curves) shows that
this configuration is homeomorphic to the open unit disk with $N$ pairwise
disjoint closed disks removed.  In particular, $D_\eps$ is connected.
Furthermore,
\[
 \R^2\setminus D_\eps
 =\bigl(\R^2\setminus G_\eps\bigr)\mathbin{\dot\cup}
   \bigcup_{j=1}^N\overline{H_{j,\eps}},
\]
so $\R^2\setminus D_\eps$ has exactly $N$ bounded connected components.
Together with its unbounded component, this gives exactly $N+1$
components of $\widehat{\C}\setminus D_\eps$.
The defining charts also give the $D_N$-symmetry of $D_\eps$.

Equation \eqref{eq:135} says that the relative stream function is
constant on $\pa G_\eps$ and on the fundamental inner boundary.  Rotation by
$2\pi/N$ gives the remaining inner equations.  By the boundary criterion
\eqref{eq:019}, the vortex patch $\one_{D_\eps}$ is a
$V$-state of angular velocity $\Om_\eps$.
\begingroup
We briefly verify the passage from the boundary equation to the weak Euler
solution.  Put $D_\eps(t)=R_{\Om_\eps t}D_\eps$ and
$V(x)=\Om_\eps x^\perp$.  Rotational covariance of the Biot--Savart law and
\eqref{eq:019} give
$V\cdot n=u_{D_\eps(t)}\cdot n$ on $\partial D_\eps(t)$.  Hence, for every
$\varphi\in C_c^\infty(\R^2)$, Reynolds' formula and
$\nabla\cdot u_{D_\eps(t)}=0$ yield
\begin{align*}
 \frac{\mathrm{d}}{\mathrm{d}t}\int_{D_\eps(t)}\varphi\,\mathrm{d}x
 &=\int_{\partial D_\eps(t)}\varphi V\cdot n\,\mathrm{d}S
  =\int_{\partial D_\eps(t)}\varphi u_{D_\eps(t)}\cdot n\,\mathrm{d}S\\
 &=\int_{D_\eps(t)}u_{D_\eps(t)}\cdot\nabla\varphi\,\mathrm{d}x.
\end{align*}
Thus $\one_{D_\eps(t)}$ satisfies the vorticity transport equation in the
distributional sense.  Since $\one_{D_\eps}\in L^1\cap L^\infty$, Yudovich
uniqueness identifies this rotating family with the global solution
\cite{Yud63}.
\endgroup
Moreover,
\[
 \one_{D_\eps}
 =\one_{G_\eps}-\sum_{j=1}^N\one_{H_{j,\eps}},
\]
so the vorticity vanishes identically in every hole; it does not merely
take a second constant value there.

The limiting angular velocity satisfies $0<\Om_*<1/2$.  Since
$\lambda_\eps$ remains bounded, the same strict inequalities hold for
$\Om_\eps$ when $\eps$ is sufficiently small.

\begingroup
Finally, define the inner boundary chart by
$p_\eps:=p_0+\rho_*^2\eps^{2N}h_\eps$.  The weighted charts and
\eqref{eq:134} show that $R_\eps$, $p_\eps$, and $\Om_\eps$
extend continuously to $\eps=0$.  Their limiting values are given by
\[
 R_\eps=\sqrt{1+2\eps^{N+2}g_\eps}\longrightarrow1,
 \qquad
 p_\eps\longrightarrow p_0
 \quad\text{in }C^{2-\alpha},
\]
and $\Om_\eps=\Om_*+\eps^{2N}\lambda_\eps\to\Om_*$.  This completes the
proof of Theorem~\ref{thm:main}.
\endgroup
\end{proof}

\begin{proof}[Proof of Proposition~\ref{prop:01}]
\begingroup
Formula \eqref{eq:main-aN} follows from \eqref{eq:g-star}.  Its positivity
follows from Proposition~\ref{prop:02}.  We now compute the boundary
expansions.  From the area chart,
\eqref{eq:g-star}, and \eqref{eq:g-rate},
\begin{align*}
 R_\eps(\theta)
 &=\sqrt{1+2\eps^{N+2}g_\eps(\theta)}\\
 &=1+a_N\eps^{N+2}\cos(N\theta)
   +O_{C^{2-\alpha}}(\eps^{N+4}).
\end{align*}
The quadratic Taylor remainder of the square root is
$O(\eps^{2N+4})$ and is therefore absorbed in the displayed error.

By \eqref{eq:134} and the definition of $X_P$ in
\eqref{eq:XP},
\[
 \|h_\eps-h_*\|_{C^{2-\alpha}}=o(1),
 \qquad
 \lambda_\eps-\lambda_*=o(1).
\]
Using the chart identities \eqref{eq:023} and
\eqref{eq:024}, together with $t_{\rm i}=\eps^{2N}$ from
\eqref{eq:075}, we obtain
\begin{align*}
 &\eps p_{t_{\rm i}h_\eps,j}(w)
  -\eps\zeta^{j-1}
   \bigl(p_0(w)+\rho_*^2\eps^{2N}h_*(w)\bigr)\\
 &\qquad
  =\rho_*^2\eps^{2N+1}\zeta^{j-1}
    \bigl(h_\eps(w)-h_*(w)\bigr)
  =o_{C^{2-\alpha}}(\eps^{2N+1}).
\end{align*}
Moreover, the definition of the physical angular velocity in
\eqref{eq:136} gives
\[
 \Om_\eps-\Om_*-\lambda_*\eps^{2N}
 =\eps^{2N}(\lambda_\eps-\lambda_*)
 =o(\eps^{2N}).
\]
This proves \eqref{eq:008}--\eqref{eq:010}.

The perturbation $h_\eps$ has neither a constant Laurent coefficient nor a
change in the leading coefficient.  Hence the normalized exterior charts
retain conformal centers $\ell\zeta^{j-1}$ and conformal radii $\rho_*$;
physical scaling gives the centers $\eps\ell\zeta^{j-1}$ and radii
$\eps\rho_*$ stated in the proposition.

For the measure expansion, the exact area chart gives
\begin{equation}\label{eq:014}
 |G_\eps\mathbin\triangle\D|
 =\frac12\int_0^{2\pi}
 \left|R_\eps(\theta)^2-1\right|\,\mathrm{d}\theta=\eps^{N+2}\int_0^{2\pi}|g_\eps(\theta)|\,\mathrm{d}\theta.
\end{equation}
For all sufficiently small $\eps$, the holes lie in $G_\eps\cap\D$, and
\[
 \sum_{j=1}^N|H_{j,\eps}|
 =\eps^2|P(t_{\rm i}h_\eps)|
 =\eps^2|P|+O(\eps^{2N+2}).
\]
Since the holes are pairwise disjoint and contained in $G_\eps\cap\D$, one
has the exact disjoint decomposition
\begin{equation}\label{eq:015}
 |D_\eps\mathbin\triangle\D|
 =|G_\eps\mathbin\triangle\D|+\sum_{j=1}^N|H_{j,\eps}|.
\end{equation}
By \eqref{eq:g-rate}, \eqref{eq:014} and the identity $g_*=a_N\cos(N\theta)$, we have
\[
 |G_\eps\mathbin\triangle\D|
 =4a_N\eps^{N+2}+O(\eps^{N+4}),
\]
because $\int_0^{2\pi}|\cos(N\theta)|\,\mathrm{d}\theta=4$.  Combining this
identity with the hole-area expansion proves
\eqref{eq:011}.

Finally, \eqref{eq:039} and \eqref{eq:seed-MN}, evaluated at
$\rho=\rho_*$, give \eqref{eq:012}.  Substitution into
\eqref{eq:071} and \eqref{eq:g-star} yields
\eqref{eq:013}, as recorded in
Remark~\ref{rem:01}.  The proof is complete.
\endgroup
\end{proof}

\begin{remark}
\label{rem:02}
An analogous problem may be posed for the Euler equation in a disk or an
annulus, with the impermeability condition on the rigid boundary.  In a disk
$B_R$, $R>1$, the Green kernel on a neighborhood of
$\overline\D\times\overline\D$ is the planar logarithmic kernel plus a smooth
image term invariant under simultaneous rotations and reflections; see
\cite{dHHM16}.  A centered Rankine vortex still induces
the solid rotation $x^\perp/2$ in its interior.  The regular part is lower
order in the leading microscopic singular block, but it contributes at order
one to the outer-scale equation.  Consequently, the outer diagonal block must
be replaced by the corresponding disk $V$-state linearization.  If that
operator is invertible in the chosen $D_N$-symmetry class,
the present two-scale structure suggests that an analogous construction
may be possible for vortex patches with several genuine holes in a disk.

In an annular fluid domain the Green kernel likewise has a local logarithmic
singularity and a smooth rotation-invariant regular part on compact subsets
\cite{HXX26b}.  The present ansatz, however, cannot be
formulated there because its polygon collapses to the origin, which is not in
the annulus.  An annular analogue would require a new seed problem: one must
select a different nondegenerate polygonal relative equilibrium and fix the
harmonic circulation around the obstacle.  It would be interesting to investigate whether analogous solutions can be
constructed in these bounded-domain settings.
\end{remark}

\section*{Acknowledgments}
The authors are grateful to Professor Jiajun Tong for bringing reference~\cite{BCGHM26} to their attention. 
Z. Xue is supported by the Postdoctoral Fellowship Program of CPSF
(Grant No.~GZB20260747) and the China Postdoctoral Science Foundation
(Grant No.~2026M793420).
W. Zhan is supported by the National Natural Science Foundation of China
(Grant No.~12571126).

\section*{Declaration on the Use of Generative AI}

The authors used OpenAI's ChatGPT (GPT-5.6 Sol) to assist with the writing,
language polishing, and checking of the manuscript. All mathematical
arguments, calculations, and proofs were independently reviewed and verified
by the authors. The authors take full responsibility for the
content, correctness, and presentation of the manuscript.

\appendix

\section{An implicit function theorem with a one-sided parameter}
\label{app:IFT}

\begin{lemma}[Continuous-parameter implicit function theorem]
\label{lem:01}
Let $X,Y$ be Banach spaces, let $U\subset X$ be open, let $x_0\in U$, and let
\[
  F:[0,a)\times U\longrightarrow Y.
\]
Assume that
\[
 F\in C([0,a)\times U;Y),
 \qquad
 D_xF\in C\bigl([0,a)\times U;\mathcal L(X,Y)\bigr),
\]
where $D_xF(t,x)\in\mathcal L(X,Y)$ denotes the bounded Fr\'echet derivative
of $F(t,\cdot)$ at $x$, and $D_xF$ is continuous in operator norm.  Suppose
further that
\[
  F(0,x_0)=0,
  \qquad
  D_xF(0,x_0):X\to Y
\]
is a bounded linear isomorphism.  Then there are
$a_0\in(0,a)$, a neighborhood
$V\subset U$ of $x_0$, and a
unique continuous map $x:[0,a_0)\to V$ such that
\[
  x(0)=x_0,
  \qquad
  F(t,x(t))=0,
  \qquad 0\le t<a_0.
\]
For each $t\in[0,a_0)$, this is the unique zero of $F(t,\cdot)$ in $V$.
If, in addition, $F|_{(0,a)\times U}$ is of class $C^1$ and
$\partial_tF$ extends continuously to $[0,a)\times U$, then
$x\in C^1([0,a_0);X)$ in the one-sided sense and
\begin{equation}\label{eq:016}
 x'(t)
 =-D_xF(t,x(t))^{-1}\partial_tF(t,x(t)),
 \qquad 0\le t<a_0,
\end{equation}
where $x'(0)$ denotes the right derivative.
\end{lemma}

\begin{proof}
Set
\[
 A:=D_xF(0,x_0)^{-1}\in\mathcal L(Y,X),
 \qquad
 T_t(x):=x-AF(t,x).
\]
By the operator-norm continuity of $D_xF$, one can choose $r>0$ and
$a_0\in(0,a)$ such that
$\overline B_r:=\overline B_X(x_0,r)\subset U$ and
\[
 \sup_{\substack{0\le t<a_0\\x\in\overline B_r}}
 \|I-AD_xF(t,x)\|_{\mathcal L(X)}\le\frac12.
\]
Since $\overline B_r$ is convex, the Banach-space fundamental theorem of
calculus gives, for $x,y\in\overline B_r$,
\[
 T_t(x)-T_t(y)
 =\int_0^1
 \bigl[I-AD_xF(t,y+s(x-y))\bigr](x-y)\,\mathrm{d}s.
\]
Thus
\begin{equation}\label{eq:017}
 \|T_t(x)-T_t(y)\|_X\le\frac12\|x-y\|_X.
\end{equation}
After reducing $a_0$ once more, continuity of $F$ and $F(0,x_0)=0$ ensure
that
\[
 \|AF(t,x_0)\|_X\le\frac r4,
 \qquad 0\le t<a_0.
\]
Since $T_t(x_0)-x_0=-AF(t,x_0)$, for every $x\in\overline B_r$ we have
\begin{align*}
 \|T_t(x)-x_0\|_X
 &\le \|T_t(x)-T_t(x_0)\|_X
      +\|T_t(x_0)-x_0\|_X\\
 &\le \frac12\|x-x_0\|_X+\frac r4
 \le\frac{3r}{4}<r.
\end{align*}
Consequently $T_t(\overline B_r)\subset\overline B_r$, uniformly for
$0\le t<a_0$.  Banach's fixed-point theorem gives a unique
$x(t)\in\overline B_r$ satisfying $T_t(x(t))=x(t)$.

It remains to verify the dependence on $t$.  For $s,t\in[0,a_0)$, insert
$T_t(x(s))$ and use \eqref{eq:017}:
\begin{align*}
 \|x(t)-x(s)\|_X
 &\le \|T_t(x(t))-T_t(x(s))\|_X
      +\|T_t(x(s))-T_s(x(s))\|_X\\
 &\le \frac12\|x(t)-x(s)\|_X
      +\|A[F(t,x(s))-F(s,x(s))]\|_X.
\end{align*}
Hence
\begin{equation}\label{eq:018}
 \|x(t)-x(s)\|_X
 \le2\|A[F(t,x(s))-F(s,x(s))]\|_X.
\end{equation}
For fixed $s$, the right-hand side tends to zero as $t\to s$ by continuity
of $F$.  Thus $t\mapsto x(t)$ is continuous, including one-sided continuity
at $t=0$.  Since $T_0(x_0)=x_0$, uniqueness gives $x(0)=x_0$.
Finally,
\[
 T_t(x)=x\quad\Longleftrightarrow\quad AF(t,x)=0
 \quad\Longleftrightarrow\quad F(t,x)=0,
\]
because $A$ is injective.  Taking $V=B_X(x_0,r)$ proves the local uniqueness
and establishes the continuous-parameter assertion.

Assume now the additional $C^1$ hypotheses.  After decreasing $a_0$ and $V$
if necessary, the continuity of $D_xF$ and the openness of the set of bounded
linear isomorphisms ensure that $D_xF(t,x)$ is an isomorphism for every
$(t,x)\in[0,a_0)\times V$.  For $t>0$, the Banach-space
implicit-function theorem, together with the local uniqueness proved above,
shows that
$x$ is $C^1$ and satisfies \eqref{eq:016}.  Set
\[
 \Phi(t):=-D_xF(t,x(t))^{-1}\partial_tF(t,x(t)).
\]
The continuity of $x$, of the two derivatives of $F$, and of inversion on the
set of isomorphisms implies that $\Phi\in C([0,a_0);X)$.  For
$0<s<t<a_0$, the fundamental theorem of calculus gives
\[
 x(t)-x(s)=\int_s^t\Phi(\tau)\,\mathrm{d}\tau.
\]
Letting $s\downarrow0$ and using $x(s)\to x_0$, we obtain
\[
 x(t)-x_0=\int_0^t\Phi(\tau)\,\mathrm{d}\tau.
\]
Consequently,
\[
 x'_+(0)=\Phi(0)
 =-D_xF(0,x_0)^{-1}\partial_tF(0,x_0),
\]
and $x'$ is continuous at $t=0$.  This proves the one-sided $C^1$ assertion.
\end{proof}

\end{document}